\documentclass[10pt]{amsart}

\usepackage{amssymb, amsmath, amsthm}
\usepackage{latexsym}
\usepackage{amsbsy}
\usepackage{amsfonts}

\usepackage{hyperref}
\usepackage{graphicx}
\usepackage{color}
\usepackage[margin=1in]{geometry} 
\usepackage[french,english]{babel}
\usepackage[utf8]{inputenc}

\allowdisplaybreaks
\DeclareSymbolFont{bbold}{U}{bbold}{m}{n}
\DeclareSymbolFontAlphabet{\mathbbold}{bbold}
\newcommand{\ind}{\mathbbold{1}}

\DeclareMathOperator\spp{Span}
\DeclareMathOperator\ds{dim}
\DeclareMathOperator\sgn{sgn}
\DeclareMathOperator\Leb{Leb}
\DeclareMathOperator\proj{Proj}
\DeclareMathOperator\erf{erf}

\theoremstyle{plain} \newtheorem{theorem}{Theorem}[section]
\theoremstyle{plain} \newtheorem{lemma}[theorem]{Lemma}
\theoremstyle{plain} \newtheorem{proposition}[theorem]{Proposition}
\theoremstyle{plain}\newtheorem{corollary}[theorem]{Corollary}
\theoremstyle{plain}\newtheorem{op}[theorem]{Open problem}
\theoremstyle{plain} \newtheorem{conj}[theorem]{Conjecture}
\theoremstyle{definition} 
\theoremstyle{remark} \newtheorem{rem}[theorem]{Remark}
\theoremstyle{remark} \newtheorem{example}[theorem]{Example}

\numberwithin{equation}{section}

\numberwithin{figure}{section}

\begin{document}
\title[The spans in Brownian motion]{The spans in Brownian motion}
\author[Steven Evans]{{Steven N.} Evans}
\thanks{Steven N. Evans is supported in part by NSF grant DMS-09-07630
and NIH grant 1R01GM109454-01.}
\address{Department of Statistics, University of California at Berkeley,
367 Evans Hall, Berkeley CA 94720-3860, USA. 
} 
\email{evans@stat.berkeley.edu}

\author[Jim Pitman]{{Jim} Pitman}
\address{Department of Statistics, University of California at Berkeley, 367 Evans Hall, Berkeley CA 94720-3860, USA. 
} 
\email{pitman@stat.berkeley.edu}

\author[Wenpin Tang]{{Wenpin} Tang}
\address{Department of Statistics, University of California at Berkeley,
367 Evans Hall, Berkeley CA 94720-3860, USA. 
} 
\email{wenpintang@stat.berkeley.edu}

\date{\today} 

\begin{abstract}
For $d \in \{1,2,3\}$, let $(B^d_t;~ t \geq 0)$ be a $d$-dimensional standard Brownian motion. We study the {\em $d$-Brownian span set} $\spp(d):=\{t-s;~ B^d_s=B^d_t~\mbox{for some}~0 \leq s \leq t\}$. We prove that almost surely the random set $\spp(d)$ is $\sigma$-compact and dense in $\mathbb{R}_{+}$. In addition, we show that $\spp(1)=\mathbb{R}_{+}$ almost surely; the Lebesgue measure of $\spp(2)$ is $0$ almost surely and its Hausdorff dimension is $1$ almost surely; and the Hausdorff dimension of $\spp(3)$ is $\frac{1}{2}$ almost surely.  We also list a number of conjectures and open problems.
\end{abstract}



\maketitle
\section{Introduction and statement of main results}
\subsection{Main results and motivation}
\label{s11}
We investigate the random set $\spp(d)$, consisting of the durations of loops at all levels in a $d$-dimensional standard Brownian motion $(B^d_t;~ t \geq 0)$ for some positive integer $d$.
That is,
\begin{equation}
\label{Spandef}
\spp(d):=\{t-s;~ B^d_t=B^d_s ~\mbox{for some}~0 \leq s \leq t\}.
\end{equation}
We call $\spp(d)$ the $d$-{\em Brownian span set}. Note that we allow loops to have
zero duration. 

Observe that $\spp(d) = \bigcup_{u > 0} \spp^{[0,u]}(d)$, where
\begin{equation}
\label{subss}
\spp^{[0,u]}(d):=\{t-s ;~ B^d_s=B^d_t~\mbox{for some}~0 \leq s \leq t  \leq u\} 
\end{equation}
is the span set of Brownian motion on $[0,u]$.
It follows from Brownian scaling that  $(\spp^{[0,u]}(d);~ u \geq 0)$ has the scaling property
\begin{equation}
\label{selfs}
(\spp^{[0, c u]}(d);~ u > 0) \stackrel{(d)}{=} (c \spp^{[0,u]}(d); ~ u > 0) ~ \mbox{for all} ~ c > 0,
\end{equation}
where $X\stackrel{(d)}{=}Y$ means that the distribution of $X$ is the same as that of $Y$. In particular,
\begin{equation}
\label{selfs2}
\spp^{[0, u]}(d) \stackrel{(d)}{=} u \spp^{[0,1]}(d) ~ \mbox{for all} ~ u > 0. 
\end{equation}
Consequently,
\begin{equation}
\label{selfs3}
\spp(d) \stackrel{(d)}{=} c \spp(d) ~ \mbox{for all} ~ c > 0.
\end{equation}  
Note that for all $u>0$, $\spp^{[0,u]}(d)$ is a random closed set, while $\spp(d)$ is a countable union of random closed sets, which may be treated as a
random Borel set, see e.g. Molchanov \cite[Section $1.2.5$]{Molbook} for background.

Given a path with values in $\mathbb{R}^d$, a point in $\mathbb{R}^d$ that is visited at least twice is called a {\em double point}, while a point visited at least $r$ times is called an {\em $r$-multiple point}.  It is a result of Kakutani \cite{Kakutani44} and Dvoretzky et al. \cite{DEK} that for $d \geq 4$, almost surely the $d$-dimensional Brownian motion does not have a double point
and hence
\begin{equation*}
\spp(d)=\{0\}~a.s. \quad \mbox{for}~d \geq 4.
\end{equation*}

In Subsection \ref{s12}, we review some known results about the multiple points of the $d$-dimensional Brownian motion for $d=1,2,3$.
As a consequence of the results reviewed there,
\begin{equation*}
\spp(d) \neq \{0\}~a.s. \quad \mbox{for}~d \in \{1,2,3\}.
\end{equation*}

We now describe our results and some of our motivations for undertaking the study of the Brownian span set.  Given $A \subset \mathbb{R}^{d}$ for some positive integer $d$, let $\Leb A$ be the Lebesgue measure of the set $A$ and $\ds_H A$ be the Hausdorff dimension of the set $A$. Our main result, which we prove in the course of the paper, is the following.

\begin{theorem} 
\label{main}
For $d=1,2,3$, almost surely the random set $\spp(d)$ is $\sigma$-compact and dense in $\mathbb{R}_{+}$. Furthermore,
\begin{enumerate}
\item
$ \spp(1) =\mathbb{R}_{+}$ a.s.,
\item
$\Leb \spp(2)=0$ a.s. and $\ds_H \spp(2)=1$ a.s., 
\item
$\ds_H \spp(3)=\frac{1}{2}$ a.s.
\end{enumerate}
\end{theorem}

  For $h>0$ let
\begin{equation}
\label{Fh}
F^{h}:= \inf\{t \geq 0;~ B^d_{t+h}=B^d_t\}
\end{equation}
be the first time at which the stationary Gaussian process $(B^d_{t+h}-B^d_t;~ t \geq 0)$ hits the origin. This Gaussian process was studied by Slepian \cite{Slepian} and Shepp \cite{SheppGaussian, Shepp} when $d=1$, see also Pitman and Tang \cite{PTacc} for further developments. Note that
\[
\spp(d) \setminus \{0\} = \{h > 0 ; ~ F^{h} < \infty\},
\]  
and so an understanding of the distributional properties of the random variables
$F^h$ is important to the study of the random set $\spp(d)$.
By Brownian scaling, the random variable $F^h$ has the same distribution as $hF$, where 
\begin{equation}
\label{bgf}
F:= F^1 = \inf\{t \geq 0;~ B^d_{t+1}=B^d_t\}.
\end{equation}
Indeed, it is even true that for $c > 0$ the stochastic process $\{F^{c h}; ~ h > 0\}$ has the same distribution as the stochastic process $\{c F^h; ~ h > 0\}$.  

A possible approach to obtaining information about $\spp(d)$ is to consider the analogous object for simple symmetric random walk.
Write $\mathbb{N}$ for the nonnegative integers and $\mathbb{N}^*$ for the positive integers.  Let $(RW_k)_{k \in \mathbb{N}}$ be a one-dimensional simple symmetric random walk.  For $n \in 2 \mathbb{N}^{*}$, put
$$F_n:=\inf\{k \geq 0;~ RW_{k+n}=RW_k\}.$$
Pitman and Tang \cite[Proposition $2.4$]{PTpattern} established the following invariance principle for the {\em first hitting bridge}
\[
\{RW_{F_n+j}-RW_{F_n}; ~ 0 \le j \le n\}.
\]

\begin{proposition} 
\cite{PTpattern}
\label{appxl}
The distribution of the process
$$\left(\frac{RW_{F_n+nu}-RW_{F_n}}{\sqrt{n}};~ 0 \leq u \leq 1 \right),
$$
where the walk is defined by linear interpolation between integer times, converges weakly to the distribution of $(B^1_{F+u}-B^1_F;~ 0 \leq u \leq 1)$ as $n \rightarrow \infty$.
\end{proposition}

 As explained in Pitman and Tang \cite[Section $2$]{PTacc}, there is an almost sure version of Proposition \ref{appxl} obtained by considering Knight's \cite{Knightapprox1} consistent embedding of simple symmetric random walks on finer and finer time and space scales in Brownian motion (see also the monograph of Knight \cite[Section $1.3$]{Knightbook} for details of this embedding). Later, Knight's approach was simplified by R\'{e}v\'{esz} \cite[Section $6.3$]{Revesz} and Szabados \cite{Sza} using what they call the {\em twist-shrinkage algorithm}. We refer to the thesis of Sz\'{e}kely \cite{Sze} for further discussions.

The proof of Proposition~\ref{appxl} relies on the fact that $F$ thought of as a map from the space of continuous real-valued paths to the nonnegative reals is continuous almost surely. As the following example shows, this map is not continuous and that makes the use of random walk approximations a more delicate matter.

\begin{example}
\label{ex2}
Define $\mathcal{C}_0[0,1]$ to be the set of continuous paths $(w_t;~0 \leq t \leq 1)$ with real values, starting from $w_0=0$. Given a path $w \in \mathcal{C}_0[0,1]$, let 
\begin{equation}
\label{12345}
\spp_w:=\{t-s;~ w_t=w_s~\mbox{for some}~0 \leq s \leq t \leq 1\}
\end{equation}
be the span set of $w$. Consider the piecewise linear function $f \in \mathcal{C}_0[0,1]$ with slopes $1$ on $[0,\frac{1}{4}] \cup [\frac{3}{4},1]$ and $-1$ on $[\frac{1}{4},\frac{3}{4}]$. It is not hard to see that $\spp_f = [0,\frac{1}{2}] \cup \{1\}$; that is, there is a gap of length $\frac{1}{2}$. For $n \in \mathbb{N}$, let $f_n \in \mathcal{C}_0[0,1]$ be the piecewise linear function with slopes $1$ on $[0,\frac{1}{4}]$, $-1$ on $[\frac{1}{4},\frac{3}{4}]$ and $1-\frac{1}{n}$ on $[\frac{3}{4},1]$. Observe that $\spp_{f_n}=[0,\frac{1}{2}]$ for each $n \in \mathbb{N}$. Define the Hausdorff distance $d_H$ between two subsets of $\mathbb{R}$ by
\begin{equation}
\label{hausdef}
d_H(X,Y):=\inf\{\epsilon \geq 0;~ X \subset Y^{\epsilon}~\mbox{and}~Y \subset X^{\epsilon}\} \quad \mbox{for}~X,Y \subset \mathbb{R},
\end{equation}
where $X^{\epsilon}$ (resp. $Y^{\epsilon}$) is the $\epsilon$-neighborhood of $X$ (resp. $Y$). It is well-known that $d_H$ is a metric on the set of all compact subsets of $\mathbb{R}$. Then for each $n \in \mathbb{N}$,
$$d_H(\spp_{f_n},\spp_f)=\frac{1}{2},$$
while $||f_n-f||_{\infty}:=\sup_{0 \leq t \leq 1}|f_n(t)-f(t)| \rightarrow 0$ as $n \rightarrow \infty$. Therefore, the map $w \mapsto \spp_w$ from $\mathcal{C}_0[0,1]$ with the sup-norm metric to compact subsets of $[0,1]$ with the Hausdorff metric is not continuous.

Observe, however, that if $\{g_n\}_{n \in \mathbb{N}}$ is a sequence in $\mathcal{C}_0[0,1]$
such that $||g_n-g||_{\infty} \rightarrow 0$ as $n \rightarrow \infty$ for some
$g \in \mathcal{C}_0[0,1]$, then 
\[
\bigcap_{m \in \mathbb{N}} \overline{\bigcup_{n > m} \spp_{g_n}} \subseteq \spp_g
\]
and any subsequential limit of $\{\spp_{g_n}\}_{n \in \mathbb{N}}$ in the Hausdorff metric
is a subset of $\spp_g$.  As this example shows, the containment may be strict.
\end{example}

We will see in Section \ref{s2} that almost surely $\spp^{[0,1]}(1)$ is a compact subset of $[0,1]$ with $0<\mathbb{E} \Leb \spp^{[0,1]}(1)<1$. In particular, 
\begin{equation}
\mathbb{P}(\spp^{[0,1]}(1) \neq \spp(1) \cap[0,1]) > 0,
\end{equation}
since $\spp(1) \cap [0,1] =[0,1]$ almost surely by Theorem \ref{main} $(1)$. 

Consider the random set
\begin{equation}
\label{ls}
L^d:=\{(s,t)\in \mathbb{R}_{+}^2;~ B^d_s=B^d_t\},
\end{equation}
The set $L^d$ can be viewed as the $0$-level set of the random field $X^d_{s,t}:=B^d_t-B^d_s$ for $s,t \geq 0$; that is, $L^d=(X^d)^{-1}(\{0\})$. For $\theta \in [-\frac{\pi}{2},\frac{\pi}{2})$, let 
\begin{equation}
\label{projection}
\proj_{\theta}: \mathbb{R}^2 \ni X \rightarrow X \cdot (\cos\theta, \sin \theta) \in \mathbb{R}
\end{equation}
be the orthogonal projection of $\mathbb{R}^2$ onto the $\theta$-direction. It is not hard to see that
\begin{equation}
\label{spproj}
\spp(d) =\sqrt{2} \proj_{-\frac{\pi}{4}}(L^d) \bigcap \mathbb{R}_{+} .
\end{equation}
The relation \eqref{spproj} suggests that the $d$-Brownian span set can be understood by studying the projection of $L^d$ onto the $-\frac{\pi}{4}$-direction. To this end, we recall a result of Rosen \cite[Theorem $6$]{Rosen83}, \cite[Theorem $2$]{Rosen84} which gives the Hausdorff dimension of the random set $L^d \setminus \{(t,t) \in \mathbb{R}_{+}^2\}$.

\begin{theorem}
\cite{Rosen83,Rosen84} \label{Rosen12}
Let $L^d$ be defined by \eqref{ls}, and the set $\mathcal{D}:=\{(t,t) \in \mathbb{R}_{+}^2\}$. Then
\begin{equation}
\label{key}
\ds_H L^d \setminus \mathcal{D}=2-\frac{d}{2}~a.s. \quad \mbox{for}~d=1,2,3.
\end{equation}
\end{theorem}

Rosen \cite{Rosen84} provided a general theory for $r$-multiple points of the $N$-parameter Brownian sheet with values in $\mathbb{R}^d$ from which the formula \eqref{key} follows as a special case by taking $r=2$, $N=1$ and $d \in \mathbb{N}^{*}$. The formula \eqref{key} for $d=2,3$ was also proved in Rosen \cite{Rosen83} by a thorough study of the {\em self-intersection local times}. We refer to Subsection \ref{s12} for a review of the theory of self-intersection local times, and connections to our problem.

Let us return to the Brownian span sets. Since the map $\proj_{-\frac{\pi}{4}}$ is Lipschitz and $\proj_{-\frac{\pi}{4}}\mathcal{D}=\{0\}$, the relation \eqref{spproj} together with Theorem \ref{Rosen12} imply the following.

\begin{corollary}
\label{upper}
\begin{equation}
\label{keyupper}
\ds_H \spp(d)\leq 2-\frac{d}{2}~a.s. \quad \mbox{for}~d=2,3.
\end{equation}
\end{corollary}

 The bound $\ds_H \spp(2) \leq 1$ is immediate, but the upper bound for $\ds_H \spp(3)$ provided by Corollary \ref{upper} is non-trivial.  One of the main contributions of this work is to prove the corresponding lower bounds
\begin{equation}
\label{lower}
\ds_H \spp(d) \geq 2-\frac{d}{2}~a.s. \quad \mbox{for}~d=2,3.
\end{equation}

Our approach is to construct a random measure $M_d(\cdot)$ on the Brownian span set $\spp(d)$ for $d=2,3$. We describe in Subsection \ref{s12} how this measure is related to the self-intersection local times. In Subsection \ref{s31}, we define rigorously the random measure $M_d(\cdot)$ by weak approximation as in the case of self-intersection local times. After a study of this measure in Subsection \ref{s32}, we apply Frostman's energy method to get the lower bound \eqref{lower}. 

To conclude, we explain why the claimed Hausdorff dimensions in Theorem \ref{main} are expected to be true in the light of a well-known result of Marstrand \cite{Mar} on fractal projections.
\begin{theorem} \cite{Mar}
\label{Marp}
Let $A$ be a Borel subset of $\mathbb{R}^2$. 
\begin{enumerate}
\item 
If $\ds_H A \leq 1$, then $\ds_H \proj_{\theta} A=\ds_H A$ for almost all $\theta \in [-\frac{\pi}{2},\frac{\pi}{2})$.
\item 
If $\ds_H A> 1$, then $\Leb \proj_{\theta} A >0$ for almost all $\theta \in [-\frac{\pi}{2},\frac{\pi}{2})$.
\end{enumerate}
\end{theorem}

We refer to Falconer \cite[Chapter $6$]{Falconer}, Mattila \cite[Chapter $9$]{Mattila} and a recent survey of Falconer et al \cite{FFJ} for further development on fractal projections. By Theorem \ref{Rosen12} and Theorem \ref{Marp} $(1)$, we have a.s. for almost all $\theta \in [-\frac{\pi}{2},\frac{\pi}{2})$,
$$\ds_H \proj_{\theta}L^d \setminus \mathcal{D}=2-\frac{d}{2} \quad \mbox{for}~d=2,3.$$
In other words, almost surely the above relations hold outside a set of directions having zero Lebesgue measure. Unfortunately, this result does not provide any information on the exceptional set. By Kaufman's dimension doubling theorem \cite{Kaufman}, almost surely $0$ does not belong to the exceptional set. Theorem \ref{main} implies that almost surely $-\frac{\pi}{4}$ does not belong to the exceptional set as well.
\subsection{Related problems and literature}
\label{s12}
First we provide a literature review on the multiple points of the $d$-dimensional Brownian motion for $d \leq 3$. In particular, the results imply that almost surely $\spp(d) \setminus \{0\} \neq \emptyset$ for $d \leq 3$.
\begin{itemize}
\item
For $d=1$, L\'{e}vy \cite{Levy40} proved that the linear Brownian motion is point recurrent and points are regular for themselves, so that almost surely any given point is visited uncountably many times. In fact, almost surely all points are visited uncountably many times by the linear Brownian motion. One way to see this is as follows.  Let $\phi(t):=\sqrt{2t |\log |\log t||}$. According to Perkins \cite{Perkins81}, almost surely
$$m_{\phi} (\{s \leq t;~B_t^1=x\})=\ell_t^x \quad \mbox{for all}~t \geq 0,~x \in \mathbb{R},$$
where $m_{\phi}(A)$ is the $\phi$-Hausdorff measure of a set $A$ and $\ell_t^x$ is the Brownian local times at level $x$ up to time $t$. This implies that almost surely,
\begin{align*}
&\{x \in \mathbb{R};~(B_t^1;~t \geq 0)~\mbox{visits}~x \in \mathbb{R}~\mbox{uncountably many times}\} \\
&\quad \quad \quad  \quad \quad \quad \quad \quad \quad \quad  \supset I_t:=\{x \in \mathbb{R};~\ell_t^x>0\} \quad \mbox{for each}~t>0.
\end{align*}
It is a consequence of the second Ray-Knight theorem \cite{RKR,RKK}, see Marcus and Rosen \cite[Theorem $2.7.1$]{MRbook}, that almost surely $I_t$ is an open interval for each $t>0$ and $\bigcup_{t>0}I_t=\mathbb{R}$.
\item
For $d=2$, Dvoretzky et al. \cite{Dvoretzkygerm} showed that almost surely planar Brownian paths contain points of uncountable multiplicity. Taylor \cite{Taylor} proved that that for all $r$ almost surely the set of $r$-multiple points of planar Brownian motion has Hausdorff dimension $2$, and later Wolpert \cite{Wolpert} provided an alternative proof for this result.  Adelman and Dvoretzky \cite{Adelman} generalized a result of Dvoretzky et al. \cite{DEKgen} by showing that almost surely for all positive integers $r$ there are {\em strict $r$-multiple points} which the planar Brownian motion visits exactly $r$ times (that is, there are points that are $r$-multiple points but not $(r+1)$-multiple points). In a series of papers \cite{LeGall86, LeGall87, LeGall89}, Le Gall determined the exact Hausdorff measure of the set of $r$-multiple points for each $r$, a result which also implies that there are strict $r$-multiple points fore each $r$.  Moreover, Le Gall \cite{LeGall} established the result that given any compact, totally disconnected set $K \subset \mathbb{R}_{+}$, almost surely there is a point $z \in \mathbb{R}^2$ such that the level set at $z$ has the same order type as $K$. In particular, almost surely there is a point $z \in \mathbb{R}^2$ such that the level set at $z$ is homeomorphic to the classical Cantor set. We refer to Le Gall \cite{LeGallSF} for various topics on the planar Brownian motion.
\item
For $d=3$, Dvoretzky et al. \cite{DEKT} showed that almost surely the three-dimensional Brownian motion does not have triple points. By refining the argument of Taylor \cite{Taylor}, Fristedt \cite{Fristedt} was able to prove that almost surely the set of double points of the three-dimensional Brownian motion has Hausdorff dimension $1$.
\end{itemize}
We refer to the survey of Khoshnevisan \cite{Khosh}, and M\"{o}rters and Peres \cite[Chapter $9$]{MP} for further development on intersections of Brownian paths. 

The existence of multiple points of L\'{e}vy processes was investigated by Taylor \cite{Taylor}, Hendricks \cite{Hend1,Hend2}, Hawkes \cite{Hawkes}, Evans \cite{Evans87}, and Le Gall et al. \cite{LRS}. Recently, the extent to which SLE paths self-intersect has received much attention. Rohde and Schramm \cite{RS} showed that the chordal SLE$_{\kappa}$ process is self-intersecting for $\kappa>4$. Relying on the prediction of Duplantier and Saleur \cite{DS} in the physics literature, Miller and Wu \cite{MW} provided a rigorous proof of the almost sure Hausdorff dimension of the double points of a chordal SLE$_{\kappa}$ path for $\kappa>4$. They also proved that almost surely the chordal SLE$_{\kappa}$ process does not have a triple point for $4< \kappa <8$, and the set of triple points is countable for $\kappa \geq 8$.  

Now we turn to the theory of self-intersection local times. Formally, this measure can be written as
\begin{equation}
\label{sil}
\alpha_d(x,B):=\iint_B \delta_x(B^d_t-B^d_s) \,ds dt \quad \mbox{for}~B \in \mathcal{B}(\mathbb{R}_{+}^2),
\end{equation}
where $\delta_x$ is the Dirac mass at $x \in \mathbb{R}^d$. The random measure $\alpha_2$ defined as in \eqref{sil} plays an important role in Symanzik's \cite{Symanzik} construction of Euclidean quantum field as well as the Edwards-Westwater's model \cite{Edwards, West1,West2,West3} of random polymers. Let $$p^2_{\epsilon}(z):=\frac{\exp\left(-\frac{|z|^2}{2 \epsilon}\right)}{2 \pi \epsilon} \quad \mbox{for}~z \in \mathbb{R}^2,$$
so that $(p^2_{\epsilon})_{\epsilon>0}$ converges weakly to $\delta_0$ as $\epsilon \rightarrow 0$. In an appendix to Symanzik \cite{Symanzik}, Varadhan \cite{Varadhan} showed that $\lim_{\epsilon \rightarrow 0}\iint_{[0,T]^2}p^2_{\epsilon}(B^2_t-B^2_s) \, ds dt$ is infinite, but
$$\iint_{[0,T]^2}p^2_{\epsilon}(B^2_t-B^2_s) \, ds dt-\frac{T}{2 \pi}\log \left(\frac{1}{\epsilon}\right)$$
converges in $L^2$ to an almost surely finite random variable as $\epsilon \rightarrow 0$. 

 In the $1980$s, Rosen \cite{Rosen85, Rosen86} established Tanaka-like formulae for self-intersection local times. These were used by Yor \cite{Yor85b} and Rosen \cite{Rosen86b} to explain {\em Varadhan's renormalization} for $\alpha_2(0, \cdot)$, and by Yor \cite{Yor85} to study the renormalization for $\alpha_3(0,\cdot)$. Around the same time, Le Gall \cite{LeGall85} derived the existence of $(\alpha_2(x, \cdot);~x \neq 0)$ from earlier work of Geman et al. \cite{GHR}, and proved that $x \mapsto \alpha_2(x, \cdot)- \mathbb{E}\alpha_2(x, \cdot)$ can be extended as a continuous function to $\mathbb{R}^2$. This provided an alternative explanation of Varadhan's renormalization. Furthermore, the renormalization of self-intersection local times for planar Brownian motion was explored by Dynkin \cite{Dynkin84, Dynkin85, Dynkin86b, Dynkin86, Dynkin87, Dynkin88,Dynkininvite}, Rosen \cite{Rosen86c}, Le Gall \cite{LeGall87b,LeGall90b}, Calais and Yor \cite{CalaisYor}, and Rosen and Yor \cite{RosenYor}. Le Jan \cite{LeJan88} provided a construction of self-intersection local times for Brownian motion on manifolds. Later Bass and Khoshnevisan \cite{BassKho} studied self-intersection local times by the theory of additive functionals for Markov processes, while Cadre \cite{Cadre}, Bass and Rosen \cite{BassRosen}, and Szabados \cite{SzaILT} approached by the strong invariance principle. 

The regularity of the renormalized self-intersection local times was investigated by Rosen \cite{Rosen1988cty, Rosen96, Rosen99, Rosen2001,Rosen2005}, Bertoin \cite{Bertoin}, Werner \cite{Werner}, Marcus and Rosen \cite{MCmem}, and Markowsky \cite{Marthesis, Mar08}. Watanabe \cite{Watan} and Shieh \cite{Shieh} initiated a white noise analysis via chaos expansions of self-intersection local times for Brownian motion. The approach was further developed by Nualart and Vives \cite{NualartVives}, Imkeller et al. \cite{IPV,IY}, He et al. \cite{He}, Hu \cite{Hu}, de Faria et al. \cite{Faria1,Faria2, Faria3}, Albeverio et al. \cite{Alb}, Mendon{\c{c}}a and Streit \cite{MSi}, Rezgui and Streit \cite{RS02}, Jenane et al. \cite{JHS}, Markowsky \cite{Mar12}, Bock et al. \cite{Five}, and Bornales et al. \cite{Sixteen}. The large deviation principle and the law of iterated logarithm for self-intersection local times for Brownian motion were considered by Chen and Li \cite{ChenLi}, Bass and Chen \cite{BassChen}, and Bass et al. \cite{BCR}, see also Chen \cite{Chenonline} \cite[Chapter $4$]{Chenbook} for further references.

Rosen \cite{Rosenfrac} considered renormalized self-intersection local times for fractional Brownian motion, which were further treated by Hu \cite{Hu01}, Hu and Nualart \cite{HN05,HN07}, Rezgui \cite{Rez07} , Hu et al. \cite{HNS, HNS2}, Rudenko \cite{Rudenko}, and Jung and Markowsky \cite{JM14, JM15}. The renormalization of self-intersection local times for Gaussian processes was explored by Berman \cite{Berman}, Izyumtseva \cite{Izy1,Izy2}, and Dorogovtsev and Izyumtseva \cite{DIbook,DI}. Recently, motivated by Lawler-Werner \cite{LW}'s construction of {\em Brownian loop soup}, Le Jan \cite{LeJan10, LeJanbook} studied the occupation fields induced by Poisson point processes of Markov loops, providing an interpretation of Dynkin's isomorphism \cite{Dynkiniso}. There the renormalization for the self-intersection local times of the Poisson loop ensemble was investigated. We refer to Sznitman \cite{Sznitman} for a user-friendly account, and to Lupu \cite{Lupu}, Fitzsimmons et al.  \cite{FR,FLR}, and Le Jan et al. \cite{LJMR} for various extensions.
 
Let us describe the connection between the self-intersection local times $\alpha_d(0,\cdot)$ and the random measure $M_d(\cdot)$ that we use to prove the lower bounds \eqref{lower}. Formally, the measure $M_d(\cdot)$ is defined as the image of the self-intersection local times $\alpha_d(0,\cdot)$ by the projection $\sqrt{2} \proj_{-\frac{\pi}{4}}$. That is,
\begin{equation}
\label{pushsil}
M_{d}(A):=\left[\left(\sqrt{2} \proj_{-\frac{\pi}{4}}\right)_{*}\alpha_d \right](0,A) \quad \mbox{for}~A \in \mathcal{B}(\mathbb{R}_{+}),
\end{equation}
where $f_{*}\mu$ is the push-forward of the measure $\mu$ by the map $f$. From Rosen \cite{Rosen83} and Le Gall \cite{LeGall85}'s explanation of Varadhan's renormalization as well as the computation in Subsection \ref{s31}, the random measure $M_d(\cdot)$ is $\sigma$-finite with infinite mass accumulated at $0$. We restrict $M_d(\cdot)$ to intervals away from the origin. 

The rest of the paper is organized as follows.
\begin{itemize}
\item
In Section \ref{spre}, we provide some preliminary observations of the $d$-Brownian span set for $d \in \{1,2,3\}$.
\item 
In Section \ref{s2}, we deal with the case of $d=1$. We study the properties of the span sets of linear Brownian motion, $\spp^{[0,u]}(1)$ for $u > 0$ and $\spp(1)$, and prove Theorem \ref{main} $(1)$.
\item
In Section \ref{s3}, we construct a measure supported by $\spp(d)$ for $d=2,3$, and investigate the properties of this measure. We prove Theorem \ref{main} $(2)(3)$. 
\end{itemize}
We show that the random set $\spp(d)$ is almost surely $\sigma$-compact in Subsection \ref{s21} and that it is almost surely dense in $\mathbb{R}_{+}$ in Subsection \ref{s31}. We present a selection of open problems in Subsections \ref{s22} and \ref{s33}.
\section{Basic properties of the $d$-Brownian span set for $d=1,2,3$}
\label{spre}
In this section, we study topological properties of the $d$-Brownian span set for $d \in \{1,2,3\}$. To proceed further, we require the following notation.
\begin{itemize}
\item
For $T>0$, define $\mathcal{C}_0([0,T], \mathbb{R}^d)$ to be the set of continuous paths $(w_t;~ 0 \leq  t \leq T)$ with values in $\mathbb{R}^d$ starting from $w_0=0$ on $[0,T]$, endowed with the topology of the uniform convergence.
\item
Define $\mathcal{C}_0([0,\infty),\mathbb{R}^d)$ to be the set of continuous paths $(w_t;~ t \geq 0)$ with values in $\mathbb{R}^d$ starting from $w(0)=0$ on $[0,\infty)$, endowed with the topology of uniform convergence on compacts.
 \item
Endow the space $\mathcal{C}_0([0,\infty),\mathbb{R}^d) \times \mathbb{R}_{+}$ with the product topology metrized by $\rho$, where
 \begin{equation}
 \label{prodm}
 \rho((w,h),(w',h')):=\sum_{n \in \mathbb{N}} \frac{1}{2^n} \min \left(\sup_{0 \leq t \leq n} ||w_t-w^{'}_{t}|| , 1\right) + |h-h'|,
 \end{equation}
for $(w,h), (w',h') \in \mathcal{C}_0([0,\infty),\mathbb{R}^d) \times \mathbb{R}_{+}$.
\end{itemize}

Observe that the span set of Brownian motion on $[0,u]$ can be written as
\begin{equation}
\label{ssss}
\spp^{[0,u]}(d)=\sqrt{2} \proj_{-\frac{\pi}{4}}\left(L^d \cap [0,u]^2 \right) \cap \mathbb{R}_{+},
\end{equation}
where $L^d$ is the random set defined by \eqref{ls}. The $d$-Brownian span set is the increasing limit of $\spp^{[0,u]}(d)$ as $u \rightarrow \infty$; that is,
\begin{equation}
\label{increasing}
\spp(d) = \uparrow \lim_{u \rightarrow \infty} \spp^{[0,u]}(d) =\bigcup_{k \in \mathbb{N}} \spp^{[0,k]}(d).
\end{equation}

 \begin{proposition}
\label{fsttopo}
Consider $d \in \{1,2,3\}$. For all $u > 0$, $\spp^{[0,u]}(d)$ is almost surely compact. Thus, $\spp(d)$ is almost surely $\sigma$-compact; that is, it is almost surely a countable union of compact sets.
\end{proposition}

\begin{proof} Recall that the set $L^d$ defined by \eqref{ls} is the $0$-level set of the random field $(X_{s,t}^d:=B^d_t-B^d_s;~ s,t \geq 0)$, and thus is closed almost surely. Hence, for all $u \geq 0$, almost surely $\proj_{-\frac{\pi}{4}}(L^{d} \cap [0,u]^2)$ is compact, as the continuous image of a compact set. By \eqref{ssss}, for all $u  \geq 0$, almost surely $\spp^{[0,u]}(d)$ is compact. Further, by \eqref{increasing}, almost surely $\spp(d)$ is $\sigma$-compact.
\end{proof}

We next show that the set
\begin{equation}
\label{meas}
\mathcal{T}^d:=\{(w,h) \in \mathcal{C}_0([0,\infty),\mathbb{R}^d) \times \mathbb{R}_{+};~ w_{t+h}=w_t~\mbox{for some}~t \geq 0\}
\end{equation}
is measurable with respect to the product of the Borel $\sigma$-fields.
\begin{proposition}
\label{Fsigma}
The set $\mathcal{T}^d$ defined by \eqref{meas} is an $F_{\sigma}$-set; that is, it is a countable union of closed sets for the product topology. In particular, $\mathcal{T}^d$ is measurable with respect to the product of the Borel $\sigma$-fields.
\end{proposition}

\begin{proof} Observe that the set $\mathcal{T}^d$ can be written as $\bigcup_{k,l \in \mathbb{N}} \mathcal{T}^d_{k,l}$, where
$$\mathcal{T}^d_{k,l}:=\{(w,h) \in \mathcal{C}_0([0,\infty),\mathbb{R}^d) \times \mathbb{R}_{+};~ 0 \leq h \leq l~\mbox{and}~w_{t+h}=w_t~\mbox{for some}~0 \leq t \leq k\}.$$
For each $k,l \in \mathbb{N}$, define the map
$$Q^d_{k,l}: \mathcal{C}_0([0,k+l],\mathbb{R}^d) \times[0,l] \ni (w,h) \longmapsto (t \mapsto w_{t+h}-w_t) \in \mathcal{C}([0,k],\mathbb{R}^d),$$
which is clearly continuous. Note that $\mathcal{T}^d_{k,l}=(Q^d_{k,l})^{-1}(\mathcal{A}^d_k)$, where
$$\mathcal{A}^d_k:=\{f \in \mathcal{C}([0,k],\mathbb{R}^d);~ f(t)=0~\mbox{for some}~0 \leq t \leq k\}$$
is closed. Therefore, $\mathcal{T}^d_{k,l}$ is closed with respect to the metric $\rho$, and $\mathcal{T}^d$ is a countable union of these closed sets.
\end{proof}

Recall the definitions of $F^h$, $h > 0$, and $F:= F^1$ 
from \eqref{Fh} and \eqref{bgf}, and recall that $F^h$ has
the same distribution as $h F$.  Put $F^0 \equiv 0$.
Recall also that $1$-Brownian span set can be expressed as
\begin{equation}
\label{altspp}
\spp(1) =\{h \geq  0;~ B_t^1=B_{t+h}^1~\mbox{for some}~t \geq 0\}=\{h \geq  0;~ F^h<\infty\}.
\end{equation}

Observe that 
$$\{F < \infty\} 
\supset 
\bigcup _{n \in \mathbb{N}} 
\{\exists 0 \le t \le 1;~B_{t+2n+1}^1 - B_{2n}^1 = B_{t+2n}^1 - B_{2n}^1\}.$$
That is, the random set $\{F<\infty\}$ contains a union of independent events with the common probability
\[
\mathbb{P}(\exists 0 \le t \le 1;~B_{t+1}^1 - B_{t}^1 = 0).
\]
This common probability is obviously nonzero; for example, by the intermediate value theorem and the continuity of
Brownian paths it is at least 
\[
\mathbb{P}(\sgn(B_{1}^1) \ne \sgn(B_{2}^1 - B_{1}^1)) = \frac{1}{2},
\]
where $\sgn(\cdot)$ is the sign of a real number. By the second Borel-Cantelli lemma, $F<\infty$ almost surely and hence  $F^h<\infty$ almost surely for each $h >  0$. 

Because of Proposition \ref{Fsigma}, we can apply Fubini's theorem to get
\begin{equation}
\label{expf}
\mathbb{E} \Leb (\spp(1) \cap [0,T])= \int_0^{T} \mathbb{P}(F^h<\infty) \, dh=T \quad \mbox{for}~T>0,
\end{equation}
which implies that almost surely $\Leb (\spp(1) \cap [0,T])=T$  for all $T > 0$. Thus, almost surely the $1$-Brownian span set $\spp(1)$ misses {\em almost nothing}. But Theorem \ref{main} $(1)$ is stronger, asserting that almost surely the random set $\spp(1)$ misses {\em nothing}. 

For $d\in \{1,2,3\}$, it follows from the fact that $\spp(d) \neq \{0\}$ almost surely that $\spp^{[0,u]}(d) \neq \{0\}$ with positive probability.
The next result is our first strengthening of this fact. The case $d=1$ is also a consequence of Corollary \ref{denseempty}. 

\begin{proposition}
\label{nee}
For $d \in  \{1,2,3\}$, the random set $\spp^{[0,1]}(d)$ has $0$ as an accumulation point almost surely.
\end{proposition}

\begin{proof}
Observe that 
$$\{\spp^{[0,1]}(d)~\mbox{has}~0~\mbox{as an accumulation point}\} \supset \bigcap_{\epsilon > 0} \left\{\spp^{[0,\epsilon]}(d) \ne \{0\}\right\}.$$
Let $(\mathcal{F}_t;~ t \geq 0)$ be the usual Brownian filtration, and $\mathcal{F}_{0}^{+}:=\bigcap_{u>0} \mathcal{F}_u$ be the germ $\sigma$-field. It is not hard to see that $\bigcap_{\epsilon > 0} \{\spp^{[0,\epsilon]}(d) \ne \{0\}\} \in \mathcal{F}_{0}^{+}$, and by the Blumenthal $0$-$1$ law,
$$\mathbb{P}\left(\bigcap_{\epsilon > 0} \{\spp^{[0,\epsilon]}(d) \ne \{0\} \}\right) \in \{0,1\}.$$
Note that the events $\{\spp^{[0,\epsilon]}(d) \ne \{0\}\}$ decreases as $\epsilon \downarrow 0$ and that $\mathbb{P}(\spp^{[0,\epsilon]}(d) \ne \{0\})=\mathbb{P}(\spp^{[0,1]}(d) \ne \{0\})$ for all $\epsilon>0$. Thus,
\[
\begin{split}
\mathbb{P}\left(\bigcap_{\epsilon > 0} \{\spp^{[0,\epsilon]}(d) \ne \{0\}\}\right)
& = \lim_{\epsilon > 0} \mathbb{P}(\spp^{[0,\epsilon]}(d) \ne \{0\}) \\
& = \mathbb{P}(\spp^{[0,1]}(d) \ne \{0\})
> 0, \\
\end{split}
\]
which implies that $\mathbb{P}\left(\bigcap_{\epsilon > 0} \{\spp^{[0,\epsilon]}(d) \ne \{0\} \}\right)=1$. Therefore,
$$\mathbb{P}(\spp^{[0,1]}(d)~\mbox{has}~0~\mbox{as an accumulation point}) = 1.$$
\end{proof}

\section{The $1$-Brownian span set}
\label{s2}
 This section is concerned with the span sets of linear Brownian motion; that is, $\spp(1)$ defined by \eqref{Spandef} for $d=1$ and $\spp^{[0,u]}(1)$ for $u > 0$ defined by \eqref{subss} for $d=1$.  We study their properties in Subsection \ref{s21} and prove Theorem \ref{main} $(1)$.  We present some open problems and conjectures in Subsection \ref{s22}.
\subsection{Properties of $\spp(1)$} 
\label{s21}
For $T>0$, let
$$\mathcal{E}(T):=\{w \in \mathcal{C}_0[0,T];~w(t) \neq w(T)=0~\mbox{for}~0<t<T\} $$
be the set of continuous excursions of length $T$. To prove Theorem \ref{main} $(1)$, we need the following lemma which was pointed out to us by Alexander Holroyd. 
\begin{lemma}
\label{HMS}
Given a path $w \in \mathcal{C}_0([0,\infty),\mathbb{R})$ and any level $x \in \mathbb{R}$, if there exist $u \geq 0$ and $T>0$ such that
$$(w_{u+t}-x;~0 \leq t \leq T) \in \mathcal{E}(T),$$
then the span set of $w$
$$\spp_{w}:=\{t-s;~w_s=w_t~\mbox{for}~0 \leq s \leq t\} \supset [0,T].$$
\end{lemma}
\begin{proof}
It suffices to show that for each $0<t<T$, $t \in \spp_w$. To this end, consider the function $f:  \mathbb{R}_{+} \ni s \mapsto w_{u+t+s}-w_{u+s} \in \mathbb{R}$. Obviously $f$ is continuous. Note that $f(0)=w_{u+t}-x$ and $f(T-t)=x-w_{u+T-t}$ have opposite signs. By the intermediate value theorem, there exists $0<s'<T-t$ such that $f(s')=w_{u+t+s'}-w_{u+s'}=0$, from which the result follows.
\end{proof}

\begin{proof}[Proof of Theorem \ref{main} $(1)$] Consider the excursions of linear Brownian motion away from $0$. According to It\^{o} excursion theory \cite{Itoex}, see e.g. Revuz and Yor \cite[Chapter XII]{RYbook}, for each $T>0$, almost surely there exists an excursion whose length exceeds $T$. By applying Lemma \ref{HMS} in the case of $x=0$, almost surely the $1$-Brownian span set $\spp(1)$ contains $[0,T]$ for each $T>0$. This yields the desired result.
\end{proof}

In the rest of this subsection, we focus on the span sets of linear Brownian motion on $[0,u]$ for $u \geq 0$.  By the same argument as for Theorem \ref{main} $(1)$, we can easily prove that
\begin{corollary}
\label{denseempty}
For each $u>0$,
\begin{equation}
\label{1010}
\mathbb{P}(\spp^{[0,u]}(1)\supset [0,\epsilon]~\mbox{for some}~\epsilon>0)=1.
\end{equation}
\end{corollary}

By Corollary \ref{denseempty} or Proposition \ref{nee}, almost surely $0$ is not isolated in the set $\spp^{[0,1]}(1)$. 

\begin{proposition}
\label{cutepf}
Almost surely, the closed random set $\spp^{[0,1]}(1)$ is perfect; that is, it does not have
any isolated points.
\end{proposition}

\begin{proof}  For $\delta>0$, let
\begin{multline*}
H_{\delta}:=\{t-s;~\exists 0 \leq s \le t \leq 1~\mbox{such that}~B_s^{1}=B_t^{1}~\mbox{and for all}~0 \leq u \le v \leq 1, \\
B_u^{1}=B_v^{1}~\mbox{and}~v-u \neq t-s~\Rightarrow~|(t-s)-(v-u)|>\delta\}
\end{multline*}
be the set of spans which are isolated by at least $\delta$ from others in $\spp^{[0,1]}(1)$. 
Suppose
for the sake of contradiction that $\mathbb{P}(H_{\delta} \neq \emptyset)>0$ for some $\delta>0$. 

For $0 \le t \le 1$ put
\begin{multline*}
G_\delta(t) :=
\{0 \leq s \le t;~B_s^{1}=B_t^{1}
~\mbox{and for all}~0 \leq u \le v < t, \\
\; B_u^{1}=B_v^{1} \Rightarrow~|(t-s)-(v-u)|>\delta\}
\end{multline*}
and
\[
E_{\delta}:= \{t \in [0,1] : G_\delta(t) \ne \emptyset\}.
\] 
Note that each $h \in H_\delta$ is of the form $t - s$ for some $t \in E_\delta$ and 
$s \in G_\delta(t)$ (but that the converse is not necessarily true)
and also that $1 \notin E_\delta$ almost surely.  It therefore suffices to
show that almost surely for every $t \in E_\delta$ there exist $t_n \downarrow t$ such that
$B_{t_n}^1 = B_t^1$, because this will imply that lengths in $H_\delta$ of the form
$t - s$ for some $s \in G_\delta(t)$ are the limits on the right of lengths of the form 
$t_n - s \in \spp^{[0,1]}(1)$, which contradicts the definition of $H_\delta$.

We claim that the set $E_{\delta}$ has at most $\lceil \frac{1}{\delta} \rceil$ elements.  
To see this, assume that 
$0 \leq t_1<t_2 \cdots <t_k \leq 1$ are distinct elements of $E_{\delta}$ and choose
$s_i \in G_\delta(t_i)$.  By construction,
\[
|(t_j-s_j)-(t_i-s_i)| > \delta, \quad 1 \leq i < j \leq k.
\]
It is not hard to see that
$$1 \geq \max_{1 \leq i \leq k}(t_i-s_i)-\min_{1 \leq i \leq k}(t_i-s_i)>(k-1) \delta,$$
which implies that $k < \frac{1}{\delta}+1$ and hence $k \le \lceil \frac{1}{\delta} \rceil$.

Set $\tau_{\delta}^0:=0 \in E_\delta$ and for $i \geq 1$ put
\[
\begin{split}
\tau_{\delta}^i:
& =
\inf\{t \in (\tau_{\delta}^{i-1},1]; \; G_\delta(t) \ne \emptyset\} \\
& = \inf\{t \in (\tau_{\delta}^{i-1},1];~\exists 0 \leq s \le t~\mbox{such that}~B_s^{1}=B_t^{1}~\mbox{and for all}~0 \leq u \le v < t, \\
& \quad \quad \quad \quad \quad \quad  \quad \quad \quad  \quad  \quad \quad \quad \quad \quad \quad \quad  B_u^{1}=B_v^{1}~\Rightarrow~|(t-s)-(v-u)|>\delta\}, \\
\end{split}
\]
with the convention $\inf \emptyset := \infty$. 
We have $\tau_{\delta}^i=\infty$ for 
$i>\lceil\frac{1}{\delta}\rceil$ and 
$$E_{\delta}=\{\tau_{\delta}^i;~\tau_{\delta}^i<\infty\}.$$
It is clear that each $\tau_{\delta}^i$ is a stopping time.
By the strong Markov property, on the event $\{\tau_{\delta}^i<\infty\}$, 
the process $(B_{\tau_{\delta}^i+u}^1-B_{\tau_{\delta}^i}^1;~u \geq 0)$ is a standard Brownian motion independent of $\mathcal{F}_{\tau_{\delta}^i}$. Because almost surely a standard Brownian motion returns to the origin infinitely often in any interval $[0,\epsilon]$ for $\epsilon>0$ we achieve the desired contradiction.
\end{proof}

\begin{rem}
We sketch an informative alternative proof of Proposition \ref{cutepf} which relies on the following facts. 
Consider $w \in \mathcal{C}_0[0,1]$, and recall the definition of $\spp_w$ from \eqref{12345}.
\begin{enumerate}
\item
If $\spp_w$ contains an isolated point, then there exists $x \in \mathbb{R}$ 
such that the level set $\{0 \leq t \leq 1;~w_t = x\}$ has two or more isolated points.
\item
If $s$ is an isolated point of the level set $\{0 \leq t \leq 1;~w_t = x\}$ for some $x\in \mathbb{R}$, then $s$ is either a local minimum, a local maximum, a point of increase or a point of decrease.
\end{enumerate}

To prove $(1)$, assume that $\spp_w$ contains an isolated point $\ell > 0$. Then there exist $\delta > 0$ and $0 \le s < t \le 1$ with $w_s = w_t=x$ and $t - s = \ell$ such that for all $0 \le u < v \le 1$, $w_u=w_v=x$ and $v-u \neq \ell$ imply that 
$$|(v-u) - \ell| > \delta.$$
Suppose by contradiction that for each $\epsilon>0$, there exists $0 \leq r \leq 1$ such that $0<|r-s| \leq \epsilon$ and $w_r=x$. By taking $\epsilon<\min(\delta,\ell)$, $u=r$ and $v=t$, we get
$$\delta<|(v-u)-\ell|=|r-s| \leq \epsilon< \delta,$$
which leads to a contradiction. Similarly, it cannot be the case that for each $\epsilon > 0$ there exists $0 \leq r \leq 1$ such that $0<|r-t| \leq \epsilon$ and $w_r=x$. Thus, $s$ and $t$ are isolated points of the level set at $x$.

To prove $(2)$, consider the case $0<s<1$. The case of endpoints can be handled similarly. By assumption, there exists $\epsilon > 0$ such that $w_t \ne x$ for $t \in (s-\epsilon,s) \cup (s, s+\epsilon)$. By path continuity, the sign of $w_t - x$ is some constant $\sigma_-  \in \{-1,+1\} $ for $s - \epsilon < t < s $ and some constant $\sigma_+  \in \{-1,+1\}$ for $s < t < s + \epsilon$.  
Now if $\sigma_- = \sigma_+=+1$ then $s$ is a local minimum; if $\sigma_- =\sigma_+= -1$, then $s$ is a local maximum; if $\sigma_- = -1$ and $\sigma_+ = +1$, then $s$ is a point of increase; 
and if $\sigma_- = +1$ and $\sigma_+ = -1$, then $s$ is a point of decrease.

It is a result of Dvoretzky et al. \cite{DEK61} 
that almost surely a Brownian path has no points of increase or decrease; 
see also Adelman \cite{AdelmanDEK}, Karatzas and Shreve \cite[Section $6.4$B]{KS}, 
Burdzy \cite{Burdzy90}, and Peres \cite{Peres96} for shorter proofs. 
Consequently, isolated points in a level set of Brownian motion over $[0,1]$ 
are necessarily local minima or local maxima. We complete the proof by showing that
\begin{enumerate}
\item[(3)]
Almost surely, there do not exist $0 \le s < t \le 1$ such that $B_s^1 = B_t^1$ and both $s$ and $t$ are local extrema.
\end{enumerate}

For $0 \leq a<b$, let $M_{ab}$ be the set of levels of local extrema of Brownian motion in $[a,b]$. It suffices to prove that for rationals $0 \le p < q < r < s \le 1$,
$$M_{pq} \cap M_{rs} = \emptyset~a.s.$$
Note that for any function $f: \mathbb{R} \rightarrow \mathbb{R}$ the set of levels of local extrema is countable, see e.g. van Rooij and Schikhof \cite[Theorem $7.2$]{vRS}. Thus, $M_{pq}$ and $M_{rs}$ are both countable. Write
$$M_{rs} = B_q^1 + (B_r^1 - B_q^1) + M_{rs}',$$
where $M_{rs}'$ is the set of levels of local extrema of Brownian motion $(B_{r+t}^1 - B_r^1;~ t \ge 0)$ in the interval $[0,s-r]$. Therefore,
$$M_{pq} \cap M_{rs} \ne \emptyset \quad \mbox{if and only if} \quad B_r^1 - B_q^1 \in M_{pq} - (B_q^1 + M_{rs}').$$

The random set $M_{pq} - (B_q^1 + M_{rs}')$ is countable as the Minkowski difference of two countable sets
and it is independent of the random variable $B_r^1 - B_q^1$. The random variable
$B_r^1 - B_q^1$ has a distribution that is absolutely continuous with respect to Lebesgue measure
 and so it has zero probability of taking values 
in a countable set. It follows from Fubini's theorem that $\mathbb{P}(M_{pq} \cap M_{rs} \ne \emptyset) = 0$,
as required.
\end{rem}

In view of the expression \eqref{altspp}, the random set $\spp^{[0,u]}(1)$ can be written as
\begin{equation}
\label{bb}
\spp^{[0,u]}(1)=\{h \in [0,u];~ F^h \leq u-h\},
\end{equation}
where $F^h$ is defined by \eqref{Fh}. 
Let 
\begin{equation}
\label{m}
S_u:=\Leb \spp^{[0,u]}(1) \quad \mbox{for}~u > 0.
\end{equation}
By \eqref{selfs}, 
\[
(S_{cu};~ u > 0) \stackrel{(d)}{=} (c S_{u};~ u > 0) ~ \mbox{for all}~c > 0.
\]
In particular, $S_u \stackrel{(d)}{=}u S_1$ for all $u > 0$.  It follows from Corollary \ref{denseempty} that for all $u>0$, $S_u>0$ almost surely. In the following proposition, we provide an estimate for the expected value of $S_1$.

\begin{proposition}
\label{2015}
For each $u>0$, $\mathbb{E}S_u=u \mathbb{E}S_1$, and $0.655 \leq \mathbb{E}S_1 \leq 0.746$.
\end{proposition}

\begin{proof} 
From the representation \eqref{bb}, we have
\begin{equation}
\label{nn}
\mathbb{E} S_1 
= \mathbb{E} \int_0^1 1\left(F^h \leq 1-h\right) \, dh
=\mathbb{E}\int_0^1 1\left(F \leq \frac{1-h}{h}\right) \, dh
=\mathbb{E}\left(\frac{1}{1+F}\right),
\end{equation}
where $F^h$ (resp. $F$) is defined by \eqref{Fh} (resp. \eqref{bgf}). It was proved in Pitman and Tang \cite[Proposition $4.1$]{PTacc} that 
$$\mathbb{P}(F \in dt)=\frac{1}{\pi}\sqrt{\frac{2-t}{t}} \quad \mbox{for}~0 \leq t \leq 1 \quad \mbox{and} \quad \mathbb{P}(F>1)=\frac{1}{2}-\frac{1}{\pi}.$$
As a consequence,
\begin{equation}
\int_0^1 \frac{1}{\pi(1+t)}\sqrt{\frac{2-t}{t}} \, dt 
\leq \mathbb{E}\left(\frac{1}{1+F} \right) 
\leq \int_0^1 \frac{1}{\pi(1+t)}\sqrt{\frac{2-t}{t}} \, dt+\frac{1}{2}\mathbb{P}(F>1),
\end{equation}
which provides the numerical bound.
\end{proof}

We refer to Pitman and Tang \cite[Section $3$]{PTacc} for further discussion on the distribution of the first hitting time $F$ defined by \eqref{bgf}.
\subsection{Some open problems}
\label{s22}
Let us consider the random set $\spp^{[0,1]}(1)$; that is the span set of linear Brownian motion on $[0,1]$. It follows from Lemma \ref{HMS} that
\begin{equation}
\label{tv}
\spp^{[0,1]}(1) \supset [0,R],
\end{equation}
where $R$ is the length of the longest complete excursion from all levels in linear Brownian motion up to time $1$. Recall that $S_1:=\Leb \spp^{[0,1]}(1)$, and let 
\begin{equation}
T_1:=\inf\{t > 0;~t \notin \spp^{[0,1]}(1)\}.
\end{equation}

According to Corollary \ref{denseempty}, neither the distribution of $S_1$ nor that of $T_1$ has any atom at $0$. By \eqref{tv}, $S_1$ and $T_1$ are at least $R$ almost surely. Thus the expectation of $R$ provides a lower bound for those of $S_1$ and $T_1$. However, the study of these random variables seems to be challenging.

\begin{op}
\begin{enumerate}
\item
Is the distribution of $S_1$ (resp. $T_1$, $R$) diffuse?
\item
Is the distribution of $S_1$ (resp. $T_1$, $R$) absolutely continuous with respect to Lebesgue measure on $[0,1]$ and, if it is, what is the Radon-Nikodym derivative?
\end{enumerate}
\end{op}

Let $R^0_1$ be the length of the longest complete excursion away from $0$ in linear Brownian motion up to time $1$. Note that it also gives a lower bound for $S_1$ and $T_1$. Let 
$$g_1:=\sup\{t<1;~B^1_t=0\}$$
be the time of last exit from $0$ on the unit interval. A result of L\'{e}vy \cite{Levy40b}, see e.g. Revuz and Yor \cite[Exercise $3.8$, Chapter XII]{RYbook}, shows that 
\begin{itemize}
\item
$g_1$ is arcsine distributed,
\item 
$(B^1_{tg_1}/\sqrt{g_1};~0 \leq t \leq 1)$ is independent of $g_1$, and has the same distribution as a standard Brownian bridge.
\end{itemize}
Therefore, $$R_1^0 \stackrel{(d)}{=} R^{0,br}_1 g_1,$$
 where $R^{0,br}_1$ is the length of the longest excursion away from $0$ in a standard Brownian bridge up to time $1$, independent of $g_1$. Pitman and Yor \cite{PT97} proved that $R^{0,br}_1$ is the first component of a sequence with the {\em Poisson-Dirichlet} $(\frac{1}{2},\frac{1}{2})$ distribution. It follows as a special case of Pitman and Yor \cite[Proposition $17$]{PT97} by taking $\alpha=\theta=\frac{1}{2}$ and $p=n=1$ that
$$\mathbb{E}R^{0,br}_1=\sqrt{\pi} \int_0^\infty \frac{\sqrt{t}e^{-t}}{[e^{-t}+\erf(t)]^2}dt \approx 0.5739,$$
where $\erf(t):=\frac{2}{\sqrt{\pi}}\int_{-\infty}^{t}e^{-s^2}ds$ is the error function. Hence, 
\begin{equation}
\label{2016}
\mathbb{E}R_1^0=\mathbb{E}R^{0,br}_1 \mathbb{E}g_1 \approx 0.2869.
\end{equation}

The lower bound given by \eqref{2016} is less tight than that given in Proposition \ref{2015}. Contrary to the case of $\spp(1)$, it is not enough to study $\spp^{[0,1]}(1)$ by only considering excursions away from $0$ on the unit interval.

We also point out that the argument of Lemma \ref{HMS} does not just work for single excursions.  If for $0 \le p < q < r < s$ we have excursions above a level $x$ over the intervals $(p,q)$ and $(r,s)$, then for all $t \geq 0$ such that 
$$p + t \in [r,s] \quad \mbox{and} \quad r - t \in [p,q]$$
or
$$q+ t \in [r,s] \quad \mbox{and} \quad s - t \in [p,q],$$
we have a span of length $t$.  That is, we have a span for all $t \geq 0$ such that
$$\max\{r - p, s - q\} \le t \le s - p \quad \mbox{or} \quad r-q \leq t \leq \min\{r-p, s-q\}.$$

Recall that the span set of the piecewise linear function $f$ in Example \ref{ex2} is a disjoint union of finitely many closed intervals. The following example shows that a continuous function of finite length can have a span set as a disjoint union of infinitely many closed intervals.
\begin{example}
\label{123455}
Consider the piecewise linear function $g \in \mathcal{C}_0[0,6]$ with slopes $1$ on $[0,1] \cup [3,\frac{17}{4}] \cup (\cup_{n \geq 2}[6-\frac{5}{2^n},6-\frac{7}{2^{n+1}}])$ and $-1$ on $[1,3] \cup [\frac{17}{4},\frac{19}{4}] \cup (\cup_{n \geq 2} [6-\frac{7}{2^{n+1}},6-\frac{5}{2^{n+1}}])$, shown in Figure \ref{fig:2}. Note that $g$ is composed of consecutive positivie/negative tent functions of heights $1$, $\frac{1}{2^2}$, $\frac{1}{2^3} \cdots$. For geometric reasons, we only need to consider the spans obtained from positive tents; those are positive excursions away from $0$. Applying Lemma \ref{HMS} to the positive tent \fbox{$1$}, we have that $[0,2] \subset \spp_g$. Observe that the spans that we get from the positive tents \fbox{$2$}, \fbox{$3$} $\cdots$ are no larger than $2$. These spans are obviously contained in $[0,2]$. Furthermore, the refinement of Lemma \ref{HMS} for different excursions shows that the spans obtained from the positive tents \fbox{$1$} and \fbox{$n$} are:
$$\left[2,\frac{5}{2}\right] \cup \left[4,\frac{9}{2}\right]\quad \mbox{for}~n=2,$$
and
$$\left[4-\frac{1}{2^{n-3}},4-\frac{3}{2^{n-1}}\right] \cup \left[ 6-\frac{1}{2^{n-3}},6-\frac{3}{2^{n-1}}\right]\quad \mbox{for}~n\geq 3.$$
Therefore, the span set of $g$ is 
\begin{equation*}
\left[0,\frac{5}{2} \right] \cup \left(\bigcup_{n \geq 3}\left[4-\frac{1}{2^{n-3}},4-\frac{3}{2^{n-1}}\right] \right) 
\cup \left[4,\frac{9}{2}\right] \cup \left(\bigcup_{n \geq 3}\left[6-\frac{1}{2^{n-3}},6-\frac{3}{2^{n-1}}\right] \right) \cup \{6\}.
\end{equation*}
\begin{figure}
\centering
	\includegraphics[width=0.70\textwidth]{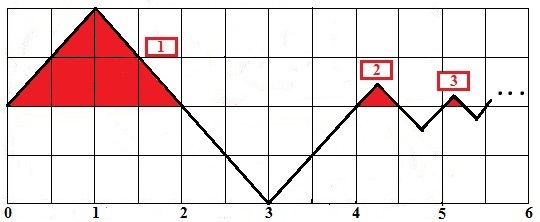}
	\caption{A path $g \in \mathcal{C}_{0}[0,6]$ whose span set is a disjoint union of infinitely many closed intervals.}
	\label{fig:2}
\end{figure}
\end{example}

The path $g$ is not ``typical'' for a linear Brownian motion. 
However,  by the support theorem, see e.g. Bass \cite[Proposition $6.5$, Chapter $1$]{Bassbook}, 
$\mathbb{P}(\sup_{0 \le t \le 6} |B_t^1 - g_t| \le \epsilon) > 0$ for any $\epsilon > 0$.
It follows that for any positive integer $k$ there is positive probability
that among the connected components of $\spp^{[0,1]}(1)$ there are at least $k$
(closed) intervals with nonempty interiors.

\begin{op}
\begin{enumerate}
\item
Is the random set $\spp^{[0,1]}(1)$ the closure of its interior almost surely?
\item
Is the random set $\spp^{[0,1]}(1)$ a disjoint union of finitely many closed intervals almost
surely?
\end{enumerate}
\end{op}

According to the Choquet-Kendall-Matheron theorem \cite{Choquet,Kendall3, Matheron},  the distribution of the random closed set $\spp^{[0,1]}(1)$ is characterized by the {\em capacity functional}.\
\begin{equation}
\label{CKM}
T_{\spp^{[0,1]}(1)}(K) := \mathbb{P}(\spp^{[0,1]}(1) \cap K \neq \emptyset) \quad \mbox{for all}~K \in \mathcal{K},
\end{equation}
where $\mathcal{K}$ is the family of compact subsets of $[0,1]$. We refer to Molchanov \cite{Molcha} \cite[Chapter $1$]{Molbook} for a review on the general theory of random closed sets.

\begin{op}
Is there an explicit closed form for the capacity functional $T_{\spp^{[0,1]}(1)}$ defined by \eqref{CKM}?
\end{op}

Alternatively, we can consider the complement of $\spp^{[0,1]}(1)$, which is almost surely open. By the Heine-Borel theorem, see e.g. Rudin \cite[Chapter $1$]{Rudin}, any open set in $\mathbb{R}$ can be expressed as a countable union of pairwise disjoint open intervals. Write
$$[0,1] \setminus \spp^{[0,1]}(1) = \bigsqcup_{k \in \mathbb{N}^{*}} \mathcal{O}_k,$$
where the open intervals $(\mathcal{O}_k)_{k \in \mathbb{N}^{*}}$ are in some arbitrary enumeration. It is interesting to understand the distribution of the ranked lengths of these open intervals.

\begin{op}
\begin{enumerate}
\item
Is the random set $[0,1] \setminus \spp^{[0,1]}(1)$ a disjoint union of finitely many open intervals almost surely or infinitely many open intervals almost surely?
\item
What is the distribution of the longest of the intervals  $(\mathcal{O}_k)_{k \in \mathbb{N}^{*}}$?
\item
More generally, what is the distribution of the ranked lengths of the intervals $(\mathcal{O}_k)_{k \in \mathbb{N}^{*}}$?
\end{enumerate}
\end{op}

It is also interesting to study the discrete version of the problem concerning the spans in a simple random walk. Let $(RW_k)_{k \in \mathbb{N}}$ be a simple random walk. Define
\begin{equation}
\spp_{RW}^N(1):=\{k-l \geq 0;~ RW_k=RW_l~\mbox{for}~k,l \leq N\}
\end{equation}
as the spans in $N$ steps of the walk. This object is of interest in its own right and there are many questions such as the following that naturally suggest themselves.

\begin{op}
What is the distribution of the number of distinct spans in $N$ steps of a random walk?
\end{op}

Recall the definition of the Hausdorff distance $d_H$ from \eqref{hausdef}. See also Burago et al. \cite[Chapter $7$]{BBI} for further development on the Hausdorff distance between compact sets. We conjecture the following result.
 
\begin{conj}
\label{approxgqq}
The sequence of random compact sets $\frac{1}{N} \spp_{RW}^{N}(1)$ converges in distribution to $\spp^{[0,1]}(1)$ with respect to the Hausdorff metric $d_H$ on the compact subsets of $[0,1]$ as $N \rightarrow \infty$.
\end{conj}

 Fix $n \in \mathbb{N}$. Let $\tau^{(n)}_0:=0$ and $\tau_{k+1}^{(n)}:=\inf\{t>\tau_k^{(n)};~|B_t^1-B^1_{\tau_k^{(n)}}|=2^{-n} \}$ for $k \in \mathbb{N}$. Note that $\left(RW^{(n)}_k:=2^nB^1_{\tau^{(n)}_k}\right)_{k \in \mathbb{N}}$ is a simple random walk. Knight \cite{Knightapprox1} proved that the sequence of linearly interpolated random walks
$$\left(\frac{RW^{(n)}_{2^{2n}t}}{2^n} ;~ t \geq 0 \right)$$
converges almost surely in $\mathcal{C}[0,\infty)$
to
$(B^1_t;~ t \geq 0)$.  Define
\begin{equation}
\spp_{KRW}^{(n)}(1):=\{k-l \geq 0;~ RW^{(n)}_k=RW^{(n)}_l~\mbox{for}~k,l \leq 2^{2n}\}
\end{equation}
Further, we conjecture that
\begin{conj}
\label{approxg}
\begin{equation}
d_H\left(\frac{1}{2^{2n}} \spp_{KRW}^{(n)}(1), \spp^{[0,1]}(1)\right) \longrightarrow 0 \quad \mbox{in probability},
\end{equation}
as $n \rightarrow \infty$.
\end{conj}

\section{The $d$-Brownian span set for $d=2,3$}
\label{s3}
 This section is devoted to the span set of $d$-dimensional Brownian motion, $\spp(d)$, for $d=2,3$. In Subsection \ref{s31}, we prove that almost surely $\spp(2)$ has null Lebesgue measure and $\spp(d)$ is dense in $\mathbb{R}_{+}$ for $d=2,3$. After recalling a general strategy for obtaining lower bounds on Hausdorff dimensions, we outline the proof of Theorem \ref{main} $(2)(3)$ by constructing a random measure $M_d$, formally defined by \eqref{pushsil} on $\spp(d)$ for $d=2,3$. In Subsection \ref{s32}, we study the measure introduced in Subsection \ref{s31} and finish the proof of Theorem \ref{main} $(2)(3)$ by establishing that this measure has finite energy in the relevant range of indices. In Subsection \ref{s33}, we present some open problems related to fractal projections.
\subsection{Hausdorff dimensions}
\label{s31}
To start with, we show that the $2$-Brownian span set is Lebesgue null.

\bigskip
\noindent
\textbf{Proof that $\Leb \spp(2)=0$ a.s.} The $2$-Brownian span set can be written as
\begin{equation}
\label{31}
\spp(2)=\{h \geq 0;~ B_t^2=B_{t+h}^2~\mbox{for some}~t \geq 0\},
\end{equation}
and, by Fubini's theorem,
\begin{equation}
\label{32}
\mathbb{E}\Leb \spp(2)=\int_0^{\infty}\mathbb{P}(B_t^2=B_{t+h}^2~\mbox{for some}~t \geq 0) \, dh.
\end{equation}
It thus suffices to show that for all $h > 0$ that  $\mathbb{P}(B_t^2=B_{t+h}^2~\mbox{for some}~t \geq 0)=0$. By Brownian scaling, it further suffices to prove that for each fixed $T>0$ that
\begin{equation}
\label{33}
\mathbb{P}(B^2_t=B^2_{t+1}~\mbox{for some}~0 \leq t \leq T)=0.
\end{equation}
Given $x \in \mathbb{R}^2$, set $B^{2,x}_t:=B^2_t-tx$ for $t \geq 0$. By the Cameron-Martin theorem, the distributions of $({B}^{2,x}_t;~ 0 \leq t \leq T)$ and $(B^2_t;~ 0 \leq t \leq T)$ are mutually absolutely continuous for all $T>0$. Thus, \eqref{33} is equivalent to 
\begin{align}
\label{34}
0&=\mathbb{P}(B^{2,x}_{t}=B^{2,x}_{t+1}~\mbox{for some}~0 \leq t \leq T)\notag\\
&=\mathbb{P}(B_{t+1}^2-B_{t}^2=x~\mbox{for some}~0 \leq t \leq T) \quad \mbox{for each fixed}~T>0.
\end{align}
Again by Fubini's theorem, it is enough to prove that
\begin{equation}
\label{35}
\Leb\{B^2_{t+1}-B^2_t;~ 0 \leq t \leq T\}=0~a.s. \quad \mbox{for each fixed}~T>0.
\end{equation}
Taking $0<T<1$, we have
\begin{align*}
\Leb \{B^2_{t+1} - B^2_t ;~ 0 \le t \le T\} &=\Leb (B^2_1 + \{(B^2_{t+1} - B^2_1) - B^2_t ;~ 0 \le t \le T\}) \\
& =\Leb \{(B^2_{t+1} - B^2_1) - B^2_t;~ 0 \le t \le T\} \\
& \stackrel{(d)}{=} 2 \, \mathrm{Leb} \{B^2_t;~ 0 \le t \le T\},
\end{align*}
since $(B^2_{t+1} - B^2_1;~ 0 \leq t \leq T)$ and $(B^2_t;~ 0 \le t \le T)$ are independent and identically distributed, and hence $((B^2_{t+1} - B^2_1) - B^2_t;~ 0 \le t \le T)$ and $(\sqrt{2} B^2_t;~ 0 \leq t \leq T)$ are identically distributed. We complete the proof by appealing to the result of L\'{e}vy \cite{Levy40} that the image of $2$-dimensional Brownian motion has Lebesgue measure $0$ almost surely, see e.g. M\"{o}rters and Peres \cite[Theorem $2.24$]{MP}.  $\square$

\bigskip

 It is clear from Theorem \ref{main} $(1)$ that almost surely $\spp(1)$ is dense in $\mathbb{R}_{+}$. For $d=2,3$ the random set $\spp(d)$ has null Lebesgue measure, but we show that  $\spp(d)$ is still dense in $\mathbb{R}_{+}$ almost surely.

\bigskip
\noindent
\textbf{Proof that $\overline{\spp(d)}=\mathbb{R}_{+}~a.s.$ for $d=2,3$.} Assume that for some $0 \leq a<b \leq \infty$,
$$\mathbb{P}(\overline{\spp(d)} \cap (a,b) \neq \emptyset)>0.$$
Then
$$\mathbb{P}(\spp^{[0,u]}(d) \cap (a,b) \neq \emptyset)>0 \quad \mbox{for some}~u>0.$$
Observe that
$$\spp(d) \supset \bigcup_{k \in \mathbb{N}} \{t-s;~B^d_s=B^d_t~\mbox{for some}~ku \leq s \leq t \leq (k+1)u\}.$$
That is, the random set $\spp(d)$ contains a union of i.i.d. copies of $\spp^{[0,u]}(d)$. Applying the second Borel-Cantelli lemma, we have
$$\mathbb{P}(\spp(d) \cap (a,b) \neq \emptyset)=1 \quad \mbox{and} \quad \mathbb{P}(\overline{\spp(d)} \cap (a,b) \neq \emptyset)=1.$$
Letting $(a,b)$ range over a suitable countable subset of intervals, we conclude that almost surely the set $\overline{\spp(d)}$ is equal to some fixed closed set $\mathcal{C}$.  By the scaling property \eqref{selfs3}, the set $\mathcal{C}$ is such that $u \mathcal{C}=\mathcal{C}$ for all $u >0$. The only closed subsets of $\mathbb{R}_{+}$ with this property are $\{0\}$ and $\mathbb{R}_{+}$. The desired result follows from the fact that $\spp(d) \neq \{0\}$ for $d\leq 3$.  $\square$

\bigskip

 From now on, we deal with the Hausdorff dimension of $d$-Brownian span set for $d=2,3$. We first recall a result of Frostman \cite{Frostman} which is useful for finding lower bounds on the Hausdorff dimension of fractal sets. 

 Given $\alpha>0$ and a measure $\mu$ on a subset $E$ of $\mathbb{R}^n$, let
\begin{equation}
\label{energy}
I_{\alpha}(\mu):=\int_E \frac{\mu(dx)\mu(dy)}{|x-y|^{\alpha}}
\end{equation}
be the {\em $\alpha$-energy} of $\mu$. {\em Frostman's energy method} is encapsulated in the following result, which can be read from M\"{o}rters and Peres \cite[Proposition $4.27$]{MP}.

\begin{theorem}
\label{em} \cite{Frostman,MP}
Let $E$ be a random set in $\mathbb{R}^n$, and $\mu$ be a non-trivial random measure supported on $E$. If $\mathbb{E}I_{\beta}(\mu)<\infty$ for every $0 \leq \beta <\alpha$, then $\ds_H E \geq \alpha$ a.s.
\end{theorem}

 Note that for a nonrandom set $S \in \mathbb{R}^d$, the converse of Theorem \ref{em} also holds true. If $\ds_H E \geq \alpha$, then there exists a non-trivial measure $\mu$ supported on $E$, and $I_{\beta}(\mu)<\infty$ for every $0 \leq \beta<\alpha$. This result is known as the {\em Frostman lemma}. For more details on the relation between $\alpha$-energy and Hausdorff dimension, we refer to the book of Mattila \cite[Chapter $8$]{Mattila}. See also M\"{o}rters and Peres \cite[Chapter $4$]{MP} for probabilistic implications, or Evans and Gariepy \cite[Chapter $4$]{EG} for applications to analysis.

 We now aim to construct a random measure $M_d$ on $\spp(d)$ for $d=2,3$ such that $\mathbb{E}I_{\alpha}(M_d)<\infty$ for every $0 \leq \alpha < 2-\frac{2}{d}$. Then, by applying Theorem \ref{energy}, we obtain the desired lower bound for $\ds_H \spp(d)$ when $d=2,3$. Let $\xi$ be exponentially distributed with rate $1$, independent of $(B^d_t;~ t \geq 0)$. We consider the random measure defined by
\begin{equation}
\label{epM}
M_{d,\epsilon}(A):=\iint_{0 \leq s \leq t \leq \xi}\ind(t-s \in A;~ |B^d_t-B^d_s| \leq \epsilon) \, ds dt,
\end{equation}
for $\epsilon>0$ and $A \in \mathcal{B}(\mathbb{R}_{+})$.

 We expect that there are suitable constants $c_{d,\epsilon}$ such that $\{c_{d,\epsilon}M_{d,\epsilon}\}_{\epsilon > 0}$ converges vaguely in probability to a non-trivial random measure $M_d$ as $\epsilon \rightarrow 0$.  A similar idea appeared earlier in the work of Rosen \cite{Rosen83, Rosen86} to shed light on Varadhan's renormalization for self intersection local times of planar Brownian motion, and in that of Le Gall \cite{LeGall} to make rigorous the intuition that between the two time instants when it hits a double point, the planar Brownian motion behaves like a Brownian bridge. Since the objective in those papers is different from ours, the random measure $M_{d,\epsilon}$ defined by \eqref{epM} seems to be new, and the computation is also more involved.

 To have some idea about the right choice of $c_{d,\epsilon}$, we are led to calculate the expectation of $M_{d,\epsilon}([a,\infty))$ for $a>0$. Let $p^d_t(x,y)$ be the transition density for $d$-dimensional Brownian motion; that is,
$$p^d_t(x,y):=(2 \pi t)^{-\frac{d}{2}} \exp\left( -\frac{|x-y|^2}{2t}\right).$$
Write $p_t(x)$ for $p_t(0,x)$. We have
\begin{align}
\mathbb{E}[M_{d,\epsilon}([a,\infty))] 
&=\mathbb{E} \int_{0 \leq s \leq t \leq \xi} \ind(t-s \geq a;~ |B^d_t-B^d_s| \leq \epsilon) ds dt \notag\\
&=\int_{x,y \in \mathbb{R}}\int_{s=0}^{\infty} e^{-s}p^d_s(x) \, dx ds 
\int_{t=a}^{\infty}e^{-t}p^d_t(y)\ind(|y| \leq \epsilon) \, dy dt \notag\\
&=\int_a^{\infty} e^{-t} (2 \pi t)^{-\frac{d}{2}} \int_{|y| \leq \epsilon} \exp\left(-\frac{|y|^2}{2t} \right) \, dy dt  \notag\\
& \label{expgod}\sim \frac{\epsilon^d}{2^{\frac{d}{2}}\Gamma(\frac{d}{2}+1)} \int_a^{\infty} e^{-t}t^{-\frac{d}{2}} \, dt \quad \mbox{as}~\epsilon \rightarrow 0,
\end{align}
where $\Gamma(t):=\int_0^{\infty}x^{t-1}e^{-x}dx$ is the Gamma function. The above computation suggests that the right scaling for $M_{d,\epsilon}$ be $c_{d,\epsilon}=\epsilon^{-d}$. Moreover, it follows from \eqref{expgod} that
$$\mathbb{E}[\epsilon^{-d} M_{d,\epsilon}([a,\infty))]=\frac{1}{2^{\frac{d}{2}}\Gamma(\frac{d}{2}+1)}\int_a^{\infty} e^{-t}t^{-\frac{d}{2}}dt \rightarrow \infty \quad \mbox{as}~a \rightarrow 0.$$
Thus, $\mathbb{E}[M_d([a,\infty)]<\infty$ for every $a>0$ and $\mathbb{E}[M_d([0,\infty)]=\infty$. This means that the limiting measure $\mathbb{E}[M_d(da)]$ is $\sigma$-finite with mass piling up in neighborhoods of $0$. Keeping the above picture in mind, we have the following result.

\begin{theorem}
\label{mainbis}
For $d=2,3$, the sequence of measures $\{\epsilon^{-d} M_{d,\epsilon}\}_{\epsilon > 0}$ converges vaguely to a $\sigma$-finite measure $M_d$ in probability; that is, for all continuous functions with compact support $f \in \mathcal{C}_c(0,\infty)$,
$$\frac{1}{\epsilon^d}\int_{\mathbb{R}_{+}^{*}} f(x)M_{d,\epsilon}(dx) \longrightarrow \int_{\mathbb{R}_{+}^{*}} f(x)M_{d}(dx) \quad \mbox{in probability},$$
as $\epsilon \rightarrow 0$. Moreover, for every $l>0$ and $0 \leq \alpha < 2-\frac{d}{2}$,
\begin{equation}
\label{3imp}
\mathbb{E}\int_{a,b \geq l} \frac{M_d(da)M_d(db)}{|a-b|^{\alpha}} <\infty.
\end{equation}
\end{theorem}

We defer the proof of Theorem \ref{mainbis} to Subsection \ref{s32}, but let us  describe briefly how it proceeds. For $a,b>0$, we consider the second moment
\begin{equation}
\label{asym}
\mathbb{E}[M_{d,\epsilon}([a,\infty)) M_{d,\delta}([b,\infty))].
\end{equation}
If we can show that $\mathbb{E}[\epsilon^{-d} M_{d,\epsilon}([a,\infty)) \cdot \delta^{-d}M_{d,\delta}([b,\infty))] $ converges as $\epsilon,\delta \rightarrow 0$ to a quantity depending on $a,b$, then $\{\epsilon^{-d}M_{d,\epsilon}([a,\infty))\}_{\epsilon>0}$ is a Cauchy sequence in $L^2$. Thus, for every $a>0$, $\{\epsilon^{-d}M_{d,\epsilon}([a,\infty))\}_{\epsilon>0}$ converges in $L^2$. This implies the vague convergence in probability as $\epsilon \downarrow 0$ of the family of random measures $\{\epsilon^{-d}M_{d,\epsilon}\}_{\epsilon>0}$  to a random measure $M_d$. As a byproduct, we obtain an expression for $\mathbb{E}[M_d(da)M_d(db)]$ that gives the bound \eqref{3imp}.
\subsection{The second moment computation}
\label{s32}
Let $0<b \leq a$. As explained in Subsection \ref{s31}, we aim to evaluate the asymptotics of \eqref{asym} as $\epsilon,\delta \rightarrow 0$. Write
\begin{equation}
\label{bigintegral}
\begin{split}
&\mathbb{E}[M_{d,\epsilon}([a,\infty)) M_{d,\delta}([b,\infty))] \\
&\quad =\mathbb{E} \int_{0 \leq s' \leq t' \leq \xi, 0 \leq u' \leq v' \leq \xi}\ind(t'-s' \geq a, v'-u' \geq b; \\
&\quad \quad \quad \quad \quad \quad \quad \quad \quad \quad \quad \quad \quad  |B^d_{t'}-B^d_{s'}| \leq \epsilon, |B^d_{v'}-B^d_{u'}| \leq \delta) \, ds'dt'du'dv'. \\
\end{split}
\end{equation}
Let us split the integral \eqref{bigintegral} according to the position of $s',t',u',v'$, with the constraint $s' \leq t'$ and $u' \leq v'$.

\noindent
{\bf Case $1$}: $0 \leq s' \leq t' \leq u' \leq v' \leq \xi$, see Figure~\ref{fig:3}. We have
\begin{equation}
\label{41}
\begin{split}
&\mathbb{E} \int_{0 \leq s' \leq t' \leq u' \leq v' \leq \xi}\ind(t'-s' \geq a, v'-u' \geq b; \\
& \quad \quad \quad \quad \quad \quad  \quad \quad \quad |B^d_{t'}-B^d_{s'}| \leq \epsilon, |B^d_{v'}-B^d_{u'}| \leq \delta) \, ds'dt'du'dv' \\
& \qquad = \int_{x,y,z,w \in \mathbb{R}}\int_{s=0}^{\infty}e^{-s}p_s(x) \, dxds \int_{t=a}^{\infty} e^{-t}p_t(y)\ind(|y| \leq \epsilon) \, dy dt\\
& \quad \quad \quad \quad \quad \quad~  \times\int_{u=0}^{\infty}e^{-u}p_u(z)dzdu \int_{v=b}^{\infty} e^{-v}p_v(w)\ind(w \leq \delta) \, dw dv \\
& \qquad \sim \frac{\epsilon^d \delta^d}{2^d \Gamma(\frac{d}{2}+1)^2} \int_a^{\infty}e^{-t}t^{-\frac{d}{2}} \, dt \int_b^{\infty}e^{-v}v^{-\frac{d}{2}} \, dv \quad \mbox{as}~\epsilon,\delta \rightarrow 0.\\
\end{split}
\end{equation}

\begin{figure}
	\centering
		\includegraphics[width=0.50\textwidth]{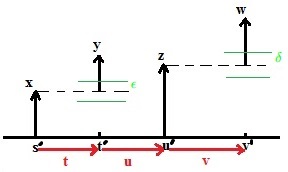}
	\caption{The case $0 \leq s \leq t \leq u \leq v$.}
	\label{fig:3}
\end{figure}

\noindent
{\bf Case $2$}: $0 \leq s' \leq u'  \leq v' \leq t' \leq \xi$, see Figure~\ref{fig:4}. We have
\begin{equation}
\label{421}
\begin{split}
&\mathbb{E} \int_{0 \leq s' \leq u' \leq v' \leq t' \leq \xi}\ind(t'-s' \geq a, v'-u' \geq b; \\
& \quad \quad \quad \quad \quad \quad \quad  \quad\quad \quad \quad \quad \quad |B^d_{t'}-B^d_{s'}| \leq \epsilon, |B^d_{v'}-B^d_{u'}| \leq \delta) \, ds'dt'du'dv'\\
&\quad=\int_{x,z,w,y \in \mathbb{R}} \int_{s=0}^{\infty}e^{-s}p_s(x) \, dx ds \Bigg\{ \int_{u=0}^{\infty}e^{-u}p_u(z) \Bigg[\int_{v=b}^{\infty} e^{-v}p_v(w)\ind(|w| \leq \delta)\\
& \quad \quad \quad  \quad \quad \quad \quad \quad \quad \int_{u+v+t \geq a} e^{-t}p_t(z+w,y)\ind(y \leq \epsilon) \, dy dt \Bigg] \, dw dv \Bigg\} \, dz du \\
& \quad \sim \frac{\epsilon^d \delta^d}{2^d \Gamma(\frac{d}{2}+1)^2} \int_{v \geq b, u+v+t \geq a}\frac{e^{-(u+v+t)}}{v^{\frac{d}{2}}(u+t)^{\frac{d}{2}}} \, du dv dt  \quad \mbox{as}~\epsilon,\delta \rightarrow 0. \\
\end{split}
\end{equation}

\begin{figure}
	\centering
		\includegraphics[width=0.50\textwidth]{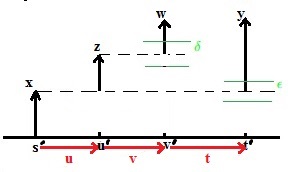}
	\caption{The case $0 \leq s \leq u \leq v \leq t$.}
	\label{fig:4}
\end{figure}

Let us make a change of variables $p=u+v+t$, $q=v$ and $r=t$. Then
\begin{equation}
\label{422}
\begin{split}
&\int_{v \geq b, u+v+t \geq a}\frac{e^{-(u+v+t)}}{v^{\frac{d}{2}}(u+t)^{\frac{d}{2}}} \, du dv dt \\
& \quad \quad \quad \quad \quad \quad \quad \quad \quad =\int_{p \geq a}\int_{b \leq q \leq p}\int_{r \leq p-q} \frac{e^{-p}}{q^{\frac{d}{2}}(p-q)^{\frac{d}{2}}} \, dr dq dp \\
& \quad \quad \quad \quad \quad \quad \quad \quad \quad= \left\{ \begin{array}{lcl}
         \int_{a}^{\infty}e^{-p} \left(\log p- \log b\right) \, dp & \mbox{for} & d=2, \\ \frac{2}{\sqrt{b}}\int_a^{\infty} e^{-p}p^{-1}(p-b)^{\frac{1}{2}} \, dp & \mbox{for} & d=3.               
\end{array}\right. \\
\end{split}
\end{equation}

\noindent
{\bf Case 3}: $0 \leq s' \leq u' \leq t' \leq v' \leq \xi$, see Figure~\ref{fig:5}. We have
\begin{equation}
\begin{split}
&\mathbb{E} \int_{0 \leq s' \leq u' \leq t' \leq v' \leq \xi}\ind(t'-s' \geq a, v'-u' \geq b; \\
&\quad \quad \quad \quad \quad \quad \quad  \quad\quad \quad \quad \quad \quad |B^d_{t'}-B^d_{s'}| \leq \epsilon, |B^d_{v'}-B^d_{u'}| \leq \delta) \, ds'dt'du'dv' \\
&=\int_{x,z,y,w \in \mathbb{R}} \int_{s=0}^{\infty}e^{-s}p_s(x) \, dx ds \Bigg\{ \int_{u=0}^{\infty}e^{-u}p_u(z) \Bigg[\int_{u+t \geq a} e^{-t}p_t(z,y)\ind(|y| \leq \epsilon) \\
& \quad \quad  \quad \quad \quad \quad \quad \quad \quad  \int_{t+v \geq b} e^{-v}p_v(y,z+w)\ind(|w| \leq \delta) \, dw dv \Bigg] \, dy dt \Bigg\} \, dz du \\
& \sim \frac{\epsilon^d \delta^d}{2^d \Gamma(\frac{d}{2}+1)^2} \int_{u+t \geq a, t+v \geq b}\frac{e^{-(u+v+t)}}{(uv+ut+vt)^{\frac{d}{2}}} \, du dv dt  \quad \mbox{as}~\epsilon,\delta \rightarrow 0. \\
\end{split}
\end{equation}

\begin{figure}
	\centering
		\includegraphics[width=0.50\textwidth]{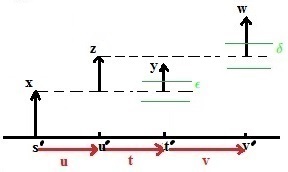}
	\caption{The case $0 \leq s \leq u \leq t \leq v$.}
	\label{fig:5}
\end{figure}

Again we make a change of variables $u+t=p$, $t+v=q$ and $t=r$. Then
\begin{equation}
\label{diffcase}
\begin{split}
&\int_{u+t \geq a, t+v \geq b}\frac{e^{-(u+v+t)}}{(uv+ut+vt)^{\frac{d}{2}}}\, du dv dt \\
& \quad \quad \quad \quad \quad \quad \quad=\int_{p \geq a,q \geq b}e^{-(p+q)}\left(\int_0^{p \wedge q}\frac{e^r}{(pq-r^2)^{\frac{d}{2}}}dr\right) \, dq dp, \\
\end{split}
\end{equation}
where $p \wedge q$ is the minimum of $p$ and $q$. We show that the RHS of \eqref{diffcase} is finite for $d=2,3$. Since $0<b \leq a$,
\begin{equation}
\label{431}
\begin{split}
& \mbox{RHS of}~\eqref{diffcase} \leq \int_{p \geq a, q \geq b} e^{-(p+q)+p \wedge q} \left( \int_0^{p \wedge q} \frac{dr}{pq-r^2}\right) \, dpdq \\
& \quad \quad \quad \quad \quad \quad=\int_{q\geq a, a \leq p<q}e^{-q} \left(\int_0^p \frac{dr}{pq-r^2}\right) \, dp dq \\
& \quad \quad \quad \quad \quad \quad \quad \quad \quad \quad~+\int_{p\geq a, b \leq q \leq p}e^{-p} \left(\int_0^q \frac{dr}{pq-r^2}\right) \, dq dp \\
\end{split}
\end{equation}

\bigskip\noindent
{\bf The case $d=2$}: 
\begin{equation}
\label{ft}
\begin{split}
&\int_{q\geq a, a \leq p<q}e^{-q} \left(\int_0^p \frac{dr}{pq-r^2}\right) \, dp dq \\
& \quad =\int_a^{\infty}\frac{e^{-q}}{\sqrt{q}}\left[\int_a^q \frac{1}{2 \sqrt{p}}\log\left(\frac{\sqrt{q}+\sqrt{p}}{\sqrt{q}-\sqrt{p}}\right) \, dp \right] \, dq \\
&\quad =\int_a^{\infty}\frac{e^{-q}}{\sqrt{q}}\Bigg[(\sqrt{q}+\sqrt{p})\log(\sqrt{q}+\sqrt{p})+(\sqrt{q}-\sqrt{p})\log(\sqrt{q}-\sqrt{p})\Bigg]_{p=a}^q \, dq \\
&\quad  =\int_a^{\infty} \frac{e^{-q}}{\sqrt{q}}[\sqrt{q}\log(4q)-(\sqrt{q}+\sqrt{a})\log(\sqrt{q}+\sqrt{a})-(\sqrt{q}-\sqrt{a})\log(\sqrt{q}-\sqrt{a})] \, dq. \\
\end{split}
\end{equation}
Similarly,
\begin{equation}
\label{st}
\begin{split}
&\int_{p\geq a, b \leq q \leq p}e^{-p} \left(\int_0^q \frac{dr}{pq-r^2}\right) \, dq dp \\
&\quad  =\int_a^{\infty} \frac{e^{-p}}{\sqrt{p}}[\sqrt{p}\log(4p)-(\sqrt{p}+\sqrt{b})\log(\sqrt{p}+\sqrt{b})-\log(\sqrt{p}-\sqrt{b})] \, dp. \\
\end{split}
\end{equation}
From \eqref{ft} and \eqref{st}, we see that the RHS of \eqref{431} and thus of \eqref{diffcase} is finite for $d=2$.

\bigskip\noindent
{\bf The case $d=3$}:
\begin{equation}
\label{ftb}
\begin{split}
& \int_{q \geq a, a \leq p <q}e^{-q}\left( \int_0^p \frac{dr}{(pq-r^2)^{\frac{3}{2}}}\right) \, dp dq  =\int_a^{\infty} \frac{e^{-q}}{q} \left[\int_a^{q} \frac{dp}{\sqrt{p(q-p)}}\right] \, dq \\
& \quad \quad \quad \quad \quad \quad \quad \quad \quad\quad \quad \quad\quad \quad \quad\quad \quad=2 \int_a^{\infty} \frac{e^{-q}}{q} \arccos \left(\sqrt{\frac{a}{q}}\right) \, dq.\\
\end{split}
\end{equation}
and
\begin{equation}
\label{stb}
\int_{p\geq a, b \leq q \leq p}e^{-p} \left(\int_0^q \frac{dr}{(pq-r^2)^{\frac{3}{2}}}\right) \, dq dp  =2 \int_a^{\infty} \frac{e^{-p}}{p} \arccos\left(\sqrt{\frac{b}{p}}\right) \, dp. 
\end{equation}
From \eqref{ftb} and \eqref{stb}, we see that the RHS of \eqref{431} and thus of \eqref{diffcase} is finite for $d=3$. 

 Note that the case $0 \leq u' \leq v' \leq s' \leq t' \leq \xi$ is similar to {\bf Case $1$} of $0 \leq s' \leq t' \leq u' \leq v' \leq \xi$, and the case $0 \leq u' \leq s' \leq v' \leq t' \leq \xi$ is similar to {\bf Case $3$} of $0 \leq s' \leq u' \leq t' \leq v' \leq \xi$. The assumption $0< b \leq a$ excludes the possibility of $0 \leq u' \leq s' \leq t' \leq v' \leq \xi$. 

 Therefore, for $0 < b \leq a$, $\mathbb{E}[\epsilon^{-d}M_{d,\epsilon}([a,\infty)) \cdot \delta^{-d}M_{d,\delta}([b,\infty))]$ converges as $\epsilon, \delta \rightarrow 0$ to
\begin{equation}
\label{tt2}
\begin{split}
&\frac{1}{4}\Bigg[2 \int_a^{\infty} \frac{e^{-p}}{p}dp \int_{b}^{\infty}\frac{e^{-q}}{q} \, dq +\int_a^{\infty}e^{-p}\log\left(\frac{p}{b}\right) \, dp \\
&\quad \quad \quad \quad \quad \quad \quad \quad \quad\quad \quad \quad\quad \quad \quad + \int_{D_{a,b}}e^{-(p+q)}\left(\int_{0}^{p \wedge q}\frac{e^r dr}{pq-r^2}\right) \, dp dq\Bigg] \\
\end{split}
\end{equation}
for $d=2$, and
\begin{equation}
\label{tt3}
\begin{split}
& \frac{2}{9 \pi}\Bigg[2 \int_a^{\infty} \frac{e^{-p}}{p} \, dp \int_{b}^{\infty}\frac{e^{-q}}{q}dq + \frac{2}{\sqrt{b}}\int_a^{\infty}\frac{e^{-p}}{p}\sqrt{p-b} \, dp \\
&\quad \quad \quad \quad \quad \quad \quad \quad \quad\quad \quad \quad \quad \quad~+ \int_{D_{a,b}}e^{-(p+q)}\left(\int_{0}^{p \wedge q}\frac{e^r dr}{(pq-r^2)^{\frac{3}{2}}}\right) \, dp dq\Bigg] \\
\end{split}
\end{equation}
for $d=3$, where 
\[
D_{a,b}:=\{p \geq a, q\geq b\} \cup \{p \geq b, q \geq a\}.
\]
 As explained in Subsection \ref{s31}, this implies that the sequence of measures $\{\epsilon^{-d} M_{d,\epsilon}\}_{\epsilon>0}$ converges vaguely to a $\sigma$-finite measure $M_d$ in probability for $d=2,3$. Moreover, for $l>0$ and $\alpha \geq 0$,
\begin{equation}
\label{final2}
\begin{split}
& \mathbb{E}\left[\int_{a,b \geq l}\frac{M_2(da)M_2(db)}{|a-b|^{\alpha}}\right] \\
& \quad \quad =\frac{1}{2}\int_{a,b \geq l, a \geq b} \frac{dadb}{|a-b|^{\alpha}}\left[\frac{2e^{-(a+b)}}{ab}+\frac{e^{-a}}{b}+2e^{-(a+b)}\int_0^b \frac{e^r dr}{ab-r^2}\right]  \\
& \quad \quad \leq \frac{1}{2}\int_{a \geq b \geq l} \frac{da db}{|a-b|^{\alpha}}\Bigg[\frac{2e^{-(a+b)}}{ab}+\frac{e^{-a}}{b} \\
& \quad \quad\quad \quad\quad \quad\quad \quad\quad \quad \quad \quad \quad~+\frac{2e^{-a}}{\sqrt{ab}}\left(2 \log(\sqrt{a}+\sqrt{b}) -\log(a-b)\right)\Bigg]. \\
\end{split}
\end{equation}
which is finite for every $\alpha<1$. Furthermore,
\begin{equation}
\label{final3}
\begin{split}
&\mathbb{E}\left[\int_{a,b \geq l}\frac{M_3(da)M_3(db)}{|a-b|^{\alpha}}\right] \\
& \quad \quad=\frac{4}{9 \pi}\int_{a \geq b \geq l} \frac{dadb}{|a-b|^{\alpha}}\Bigg[\frac{2e^{-(a+b)}}{ab}+\frac{e^{-a}}{\sqrt{b^3(a-b)}} \\
&\quad \quad\quad \quad\quad \quad\quad  \quad\quad \quad\quad \quad\quad \quad\quad  \quad \quad \quad~+2e^{-(a+b)}\int_0^b \frac{e^r dr}{(ab-r^2)^{\frac{3}{2}}}\Bigg]  \\
& \quad \quad \leq \frac{4}{9 \pi}\int_{a \geq b \geq l} \frac{dadb}{|a-b|^{\alpha}}\left[\frac{2e^{-(a+b)}}{ab}+\frac{e^{-a}}{\sqrt{b^3(a-b)}}+\frac{2e^{-a}}{a\sqrt{b(a-b)}}\right], \\
\end{split}
\end{equation}
which is finite for every $\alpha<\frac{1}{2}$. This finishes the proof of Theorem \ref{mainbis}.  $\square$

\bigskip
\subsection{Further open problems}
\label{s33}
As mentioned at the end of Introduction, we have
$$\ds_H \proj_{\theta_0}L^d \setminus \mathcal{D}=2-\frac{d}{2}~a.s. \quad \mbox{for}~d=2,3.$$
for $\theta_0=-\frac{\pi}{4}$ (Theorem \ref{main}), and $\theta_0=0$ or $-\frac{\pi}{2}$ (Kaufman's dimension doubling theorem \cite{Kaufman}). We hope that arguments similar to those presented here  can be used to deal with the projection in any direction, possibly with much tougher computations.

\begin{conj}
For every $\theta \in[-\frac{\pi}{2},\frac{\pi}{2})$,
$$\ds_H \proj_{\theta}L^d \setminus \mathcal{D}=2-\frac{d}{2}~a.s. \quad \mbox{for}~d=2,3.$$
\end{conj}

Further, by Marstrand's projection theorem \cite{Mar}, almost surely for almost all $\theta \in [-\frac{\pi}{2},\frac{\pi}{2})$,
$$\ds_H \proj_{\theta}L^d \setminus \mathcal{D}=2-\frac{d}{2} \quad \mbox{for}~d=2,3.$$
A general theorem of Falconer \cite{Falconer1} says that for $E \subset \mathbb{R}^2$ and $\ds_H E>1$
the exceptional set of directions satisfies
$$\ds_H\{\theta;~ \Leb \proj_{\theta}E=0\} \leq 2-\ds_H E.$$
However, this result does not apply in our case, since both $\ds_H \proj_{\theta}L^2$ and $\ds_H \proj_{\theta}L^3$ are smaller than $1$. More recently, projections and the exceptional set of directions have been investigated for specific set, where it is sometimes possible to identify the exceptional directions. For example, Peres and Shmerkin \cite{PS}, Hochman and Shmerkin \cite{HS} proved that there is no exceptional direction for {\em self-similar sets} with dense rotations. We refer to the survey of Shmerkin \cite{Shmerkin} for further development.

\begin{op}
Do we have almost surely for all $\theta \in [-\frac{\pi}{2},\frac{\pi}{2})$ that
$$\ds_H \proj_{\theta}L^d \setminus \mathcal{D}=2-\frac{d}{2} \quad \mbox{for}~d=2,3 $$
and, if not, what can we say about the exceptional set of directions?
\end{op}

Finally, the construction of the random measure $M_d$ in Subsection \ref{s31} also works in the case of $d=1$. 
For $x \in \mathbb{R}$, let
$$\Lambda^x(A):=\int_{A \cap [0,\xi)} d \ell_t^x \quad \mbox{and} \quad \widetilde{\Lambda^x}(A):=\Lambda^x(A) \quad \mbox{for}~A \in \mathcal{B}(\mathbb{R}_{+}).$$
It is not hard to see that
\begin{equation}
M_1(A) = \int_{-\infty}^{\infty} \int_0^\xi \int_s^\xi \ind(t-s \in A) \, d \ell_t^x \, d \ell_s^x \, dx
\quad \mbox{for}~A \in \mathcal{B}(\mathbb{R}_{+})
\end{equation}
and so $M_1$ is the trace of the random measure
\begin{equation}
\label{1dcase}
\int_{-\infty}^{\infty} \Lambda^x \ast \widetilde{\Lambda^x} \,dx
\end{equation}
on $\mathbb{R}_{+}$,
where $\ast$ denotes the  convolution of measures. 

\begin{proposition}
\label{consistent29}
The compactly supported random measure $M_1$  almost surely does not have a continuous density on $\mathbb{R}_{+}$.
\end{proposition}

\begin{proof} 
It suffices to show that almost surely the random measure in \eqref{1dcase} does not have a continuous density.
Observe that the Fourier transform of the latter random measure is the nonnegative function 
$$\int_{-\infty}^{+\infty} |\widehat{\Lambda^x}(\cdot)|^2  dx,$$
where $\widehat{\Lambda^x}(\cdot)$ is the Fourier transform of the measure $\Lambda^x$. A finite measure with a nonnegative Fourier transform has a bounded and continuous density if and only if the Fourier transform is integrable. However, if $\int_{-\infty}^{+\infty} |\widehat{\Lambda^x}(\cdot)|^2 dx$ was integrable, then $|\widehat{\Lambda^x}(\cdot)|^2$ would be integrable for almost every $x \in \mathbb{R}$. This, however, would imply by the Parseval identity that $\Lambda^x$ has a square-integrable density for almost every $x \in \mathbb{R}$, which contradicts the fact that $\Lambda^x$ is a nontrivial measure that is singular with respect to Lebesgue measure.
\end{proof}

By Corollary \ref{denseempty}, we know that the random set $\spp^{[0,1]}(1)$ contains intervals almost surely. It is believable that $\spp^{[0,1]}(1)$ is the closure of its interior with probability one. In this case it is presumably true that the set $\spp^{[0,1]}(1)$ is the support of a measure with a bounded and continuous density, and the obvious candidate for such a measure is the one built from Brownian local time in the same manner that the measure $M_1$ on $\spp^{[0,\xi]}(1)$ is constructed. However, Proposition \ref{consistent29} provides some evidence that this is not true.
\vskip 24pt
\noindent
\textbf{Acknowledgments.} We thank Alexander Holroyd for pointing out the direct proof of Theorem \ref{main} $(1)$. We also thank Yuval Peres for helpful discussions.
\bibliographystyle{amsplain}
\bibliography{Span}
\end{document}